\documentclass[a4paper,12pt]{article} 
\title{The moduli stack of enriched structures and a logarithmic compactification}

\usepackage{amsmath, amssymb, mathrsfs, amsthm, shorttoc, geometry, stmaryrd}
\usepackage{tikz}
\usepackage{mathtools} 
\usepackage{hyperref}
\usepackage{xcolor}
\usepackage{subcaption}
\hypersetup{
colorlinks,
linkcolor={black},
citecolor={blue!50!black},
urlcolor={blue!80!black}
}
\usepackage[all]{xy}
\usepackage{tabularx}
\usepackage{longtable}
\numberwithin{equation}{subsection}
\usepackage{enumitem}

\usepackage{pgf,tikz}
\usetikzlibrary{matrix, calc, arrows}

\usepackage{tikz-cd} 

\usepackage{todonotes}
\reversemarginpar

\makeatletter
\newcommand*{\doublerightarrow}[2]{\mathrel{
  \settowidth{\@tempdima}{$\scriptstyle#1$}
  \settowidth{\@tempdimb}{$\scriptstyle#2$}
  \ifdim\@tempdimb>\@tempdima \@tempdima=\@tempdimb\fi
  \mathop{\vcenter{
    \offinterlineskip\ialign{\hbox to\dimexpr\@tempdima+1em{##}\cr
    \rightarrowfill\cr\noalign{\kern.5ex}
    \rightarrowfill\cr}}}\limits^{\!#1}_{\!#2}}}
\newcommand*{\triplerightarrow}[1]{\mathrel{
  \settowidth{\@tempdima}{$\scriptstyle#1$}
  \mathop{\vcenter{
    \offinterlineskip\ialign{\hbox to\dimexpr\@tempdima+1em{##}\cr
    \rightarrowfill\cr\noalign{\kern.5ex}
    \rightarrowfill\cr\noalign{\kern.5ex}
    \rightarrowfill\cr}}}\limits^{\!#1}}}
\makeatother

\newcommand{\ul}[1]{{\underline{#1}}}

\newcommand{\Div}{\mathfrak{Div}}

\newcommand{\Mbar}{\overline{M}}

\newcommand{\gp}{\operatorname{gp}}

\renewcommand{\gp}{{\operatorname{gp}}}

\newcommand{\Cf}{\mathfrak{C}}
\newcommand{\Mf}{\mathfrak{M}}
\newcommand{\Uf}{\mathfrak{U}}
\newcommand{\Vf}{\mathfrak{V}}

\newcommand{\Nbar}{\overline{N}}

\newcommand{\x}{x}

\newcommand{\CFG}{\textbf{CFG}}
\newcommand{\Sch}{\textbf{Sch}}
\newcommand{\LogSch}{\textbf{LogSch}}
\newcommand{\Log}{\textbf{Log}}
\newcommand{\LC}{\textbf{LC}}

\newcommand{\RLC}{\textbf{RLC}}
\newcommand{\Min}{\textbf{Min}}
\newcommand{\ARLC}{\textbf{WRLC}}

\newcommand{\MLC}{\textbf{MLC}}

\newcommand{\et}{\text{\'et}}

\DeclareMathOperator{\Hom}{Hom}
\DeclareMathOperator{\Aut}{Aut}
\DeclareMathOperator{\Bl}{Bl}

\DeclareMathOperator{\id}{id}
\DeclareMathOperator{\Spec}{Spec}

\DeclareMathOperator{\Cuts}{Cuts}

\newcommand{\A}{\mathbb{A}}

\newcommand{\C}{\mathbb{C}}

\newcommand{\G}{\mathbb{G}}

\newcommand{\N}{\mathbb{N}}
\renewcommand{\P}{\mathbb{P}}

\newcommand{\R}{\mathbb{R}}

\newcommand{\Z}{\mathbb{Z}}

\newcommand{\Ccal}{\mathcal{C}}

\newcommand{\Lcal}{\mathcal{L}}
\newcommand{\Mcal}{\mathcal{M}}

\newcommand{\Ocal}{\mathcal{O}}

\tikzset{
    labl/.style={anchor=south, rotate=-90, inner sep=.5mm}
}

\newcommand{\sub}{\subseteq}

\theoremstyle{definition}
\newtheorem{definition}{Definition}[section]

\newtheorem{example}[definition]{Example}
\newtheorem{remark}[definition]{Remark}

\theoremstyle{plain}

\newtheorem{proposition}[definition]{Proposition}
\newtheorem{lemma}[definition]{Lemma}
\newtheorem{theorem}[definition]{Theorem}
\newtheorem{corollary}[definition]{Corollary}

\newtheorem{intheorem}{Theorem}

\theoremstyle{remark}

\usepackage{letltxmacro}
\LetLtxMacro{\phiorig}{\phi}
\renewcommand{\phi}{\varphi}

\author{Pim Spelier}
\date{\today}

\newcounter{nootje}
\newcommand{\beq}{\begin{equation}}
\newcommand{\eeq}{\end{equation}}
\newcommand{\beqs}{\begin{equation*}}
\newcommand{\eeqs}{\end{equation*}}

\tikzset{
  symbol/.style={
    draw=none,
    every to/.append style={
      edge node={node [sloped, allow upside down, auto=false]{$#1$}}}
  }
}

\begin{document}
\maketitle
\tableofcontents


\begin{abstract}
    Enriched curves have been studied over algebraically closed fields by Main\`o (\cite{maino1998}) and recently over general base schemes in \cite{biesel2019}. In this paper, we study enriched curves from a logarithmic viewpoint: we give a succinct definition of the stack of rich log curves as an open substack of the stack of log curves, and define an enriched curve to be a curve with a minimal rich log structure on it. This logarithmic view point turns out to be a natural language for enriched structures, leading to a simple modular compactification. This modular compactification is a smooth log blowup of the stack of log curves, answering affirmatively two questions from \cite{biesel2019}. We also generalise the concept of rich curves to $r$-rich curves, and show similar results. We include a chapter phrasing some of the key definitions solely in the language of real tropical geometry.
\end{abstract}

\section{Introduction}
The theory of enriched curves, or curves with enriched structures, has been studied mainly in the context of limits of linear series. Let $\Cf/k[[t]]$ be a prestable curve with smooth generic fibre, and let $\Lcal$ be a line bundle on the generic fiber. One of the basic goals is to understand the ways in which $\Lcal$ can extend to the whole of $\Cf$. In \cite{maino1998} Main\`o studies this using so-called \emph{enriched structures} on prestable curves over algebraically closed fields, as defined in the following (paraphrased) definition.
\begin{definition}[First definition of \cite{maino1998}]
\label{def:richoverk}
Let $C/\Spec k$ with $k = \overline{k}$ be a stable curve with irreducible components $C_1,\dots,C_n$. An \emph{enriched structure} on $C$ is a collection of line bundles $\Lcal_1,\dots,\Lcal_n$ on $C$ such that there exists a regular smoothing $\Cf/\Spec k[[t]]$ with $\Lcal_i \cong O_\Cf(C_i)|_C$. An \emph{enriched curve} is a stable curve together with an enriched structure.
\end{definition}

This definition cannot easily be extended to prestable curves over general bases since the notion of components breaks down. In a recent paper by Biesel and Holmes (\cite{biesel2019}), a generalisation was proposed. They first define the functor of enriched structures \'etale locally on the base of a prestable curve, using the combinatorics of the dual graph, and then apply descent. They show that this functor of enriched structures is representable by an algebraic space over $\Mf$, the stack of prestable curves, and they also define a modular compactification.

In this paper, we take a different approach to constructing these moduli spaces, through log geometry, replacing the combinatorics by tropical curves and monoids. We work with log curves, prestable curves with a log smooth log structure, throughout this paper.  We define when a log curve is rich, and define the algebraic stack with log structure $\Min(\RLC)/\Sch$ of rich log curves\footnote{In fact, $\Min(\RLC)$ consists of the \emph{minimal} rich log curves. It represents the stack $\RLC$ fibred over $\LogSch$ of rich log curves. For more details, see Section~\ref{subsec:gillam}.}. Using this machinery, results such as base changes of rich log curves being rich follow automatically. Over a geometric point, a rich log curve naturally comes with an enriched structure, and vice versa. Our main theorem consists of the following results on the stack $\Min(\RLC)$. 
\begin{intheorem}[Theorem~\ref{thm:rich}, Corollary~\ref{cor:maino}, Theorem~\ref{thm:rlcsmoothalgebraic}]
\label{thm:mainrlc}
The stack $\Min(\RLC)$ of rich log curves is a smooth algebraic stack with log structure. The forgetful map to $\mathfrak{M}$, the stack of prestable curves, is log \'etale. There is a canonical bijection between the $k$-points of the fiber of $\Min(\RLC)$ over a geometric point $\Spec k$ and the enriched curves over $\Spec k$.
\end{intheorem}
In \cite{biesel2019} the question is asked whether there is a natural smooth compactification of the stack $\Min(\RLC)$. In this paper, we answer that question affirmatively.
\begin{intheorem}[Theorem~\ref{thm:comp},Theorem~\ref{thm:smooth}]
\label{thm:mainarlc}
The stack $\Min(\RLC)/\Mf$ has a natural modular relative compactification $\Min(\ARLC)/\Mf$ consisting of \emph{weakly rich log curves}. The stack $\Min(\ARLC)$ is a smooth log blowup of $\Mf$ and the embedding $\Min(\RLC) \to \Min(\ARLC)$ is a map of log stacks.
\end{intheorem}
To illustrate the simplicity of our definitions, we directly give the definition of weakly rich log curves.
\begin{definition}
Let $C/S$ be a log curve with $S$ a geometric point, let $\Mbar_S$ be the characteristic monoid of the base, and let $\Gamma$ be the corresponding tropical curve. Then $C$ is called \emph{weakly rich} if for every cut\footnote{A \emph{cut} of a connected graph is a minimal set of edges that splits the graph into two connected components.} in the graph of $\Gamma$, there is a shortest edge in this cut. A log curve is weakly rich if it is weakly rich over all strict geometric points.
\end{definition}
\begin{example}
\label{ex:intro}
Consider the case where the underlying algebraic curve $\ul{C}/\ul{S}$ is given by three copies of $\P^1_{\ul{S}}$ with each pair intersecting once, and where $\ul{S}$ is a geometric point (see Figure~\ref{fig:sub-first}). The dual graph of $\ul{C}/\ul{S}$ is given in Figure~\ref{fig:sub-second}, with the edges labelled $E = \{e_1,e_2,e_3\}$. This graph has cuts $\{e_1,e_2\},\{e_1,e_3\},\{e_2,e_3\}$, as removing any two of the three edges cuts the graph into exactly two connected components. Then a log curve $C/S$ lying over $\ul{C}/\ul{S}$ induces three lengths $\ell(e_1),\ell(e_2),\ell(e_3)$ in the characteristic monoid $\Mbar_S$. By definition the curve $C/S$ is weakly rich if every cut of the dual graph has a smallest element. For this curve, that is equivalent to $\ell(e_1),\ell(e_2),\ell(e_3)$ being pairwise comparable. 

For example, the minimal log curve $C_0/S_0$ over $\ul{C}/\ul{S}$, which has $\Mbar_{S_0} = \bigoplus_{i=1}^3 \N\cdot e_i$ and $\ell(e_i) = e_i$, is not weakly rich. In contrast, a log curve $C_1/S_1$ having $\Mbar_{S_1} = \N^3$ and $\ell(e_1) = (1,0,0), \ell(e_2) = (1,1,0), \ell(e_3) = (1,1,1)$ is weakly rich since all of $(1,0,0),(1,1,0),(1,1,1)$ are comparable.

By Theorem~\ref{thm:mainarlc} the stack $\Min(\ARLC)_{S_0}/{S_0}$ is a log blowup of $S_0$. This log blowup corresponds, similarly to a toric blowup, to a subdivision of $\R_{\geq 0}^3$. This subdivision (or rather, an intersection of it with the plane defined by $a+b+c=1$) is shown in Figure~\ref{fig:sub-third}. This entire example is treated in more detail in Example~\ref{ex:subdiv}.

\begin{figure}[ht]
\begin{subfigure}{.32\textwidth}
  \centering
  \begin{tikzpicture}[thick,scale=0.5]

    
    
    \draw (20:4) -- (160:4);
    \draw (140:4) -- (280:4);
    \draw (260:4) -- (40:4);

	\end{tikzpicture}
  \caption{The curve $\ul{C}$}
  \label{fig:sub-first}
\end{subfigure}
\begin{subfigure}{.32\textwidth}
  \centering
  \begin{tikzpicture}[thick,scale=0.5]
    \begin{scope}[every node/.style={circle, draw,fill=black!50,inner sep=0pt, minimum width=4pt}]
    \node (A) at (90:3) {};
    \node (B) at (210:3) {};
    \node (C) at (330:3) {};
    \end{scope}

    \draw (A) -- (B);
    \draw (B) -- (C);
    \draw (C) -- (A);
    
    \node at (270:2) {$e_1$};
    \node at (150:2) {$e_2$};
    \node at (30:2) {$e_3$};
    hol

	\end{tikzpicture}
  \caption{The dual graph of $C$}
  \label{fig:sub-second}
\end{subfigure}
\begin{subfigure}{.32\textwidth}
  \centering
  \begin{tikzpicture}[thick,scale=0.5]
    \begin{scope}[every node/.style={circle, draw,fill=black!50,inner sep=0pt, minimum width=4pt}]
    \node (B) at (90:3) {};
    \node (C) at (210:3) {};
    \node (D) at (330:3) {};
    
    \draw (B) -- (C) node (BC) [midway] {};
    \draw (B) -- (D) node (BD) [midway] {};
    \draw (C) -- (D) node (CD) [midway] {};
    
    \draw (B) -- (CD);
    \draw (C) -- (BD);
    \draw (D) -- (BC);
    
    \node (A) at (0:0) {};
    \end{scope}

	\end{tikzpicture}
  \caption{The subdivision of $\R_{\geq 0}^3$ corresponding to $\Min(\ARLC)$ intersected with the plane $e_1 + e_2 + e_3 = 1$}
  \label{fig:sub-third}
\end{subfigure}
\caption{}
\label{fig:fig}
\end{figure}
\end{example}

In \cite[Section~5]{biesel2019}, the stack of enriched curves is studied as the first step in describing a universal N\'eron-model-admitting stack $\widetilde{\Mf}$ over the stack $\Mf$ of prestable curves. The stack $\widetilde{\Mf}$ is not quasi-compact, but can be written as a colimit of quasi-compact stacks $\bigcup_r \widetilde{\Mf}^{\leq r}$, with $\Mf^{\leq 1}$ coinciding with the stack of enriched curves (\cite[Theorem 6.9]{biesel2019}). We prove the following theorems on the stack $\widetilde{\Mf}$.
\begin{intheorem}[Theorem~\ref{thm:rich}, Theorem~\ref{thm:rlcsmoothalgebraic}, Lemma~\ref{lem:nmam}, Theorem~\ref{theorem:compbiesel}]
\label{thm:mainrrlc}
There is an infinite sequence of smooth algebraic log stacks $(\Min(\RLC_r))_{r \in \N_{\geq 1} \cup \{\infty\}}$ with $\Min(\RLC_1) = \Min(\RLC)$ such that \[\bigcup_{r\in\N_{\geq 1}} \Min(\RLC_r) = \Min(\RLC_\infty) \cong \widetilde{\Mf}.\] Here $\Min(\RLC_r)$ is the stack of \emph{$r$-rich log curves}. Under this isomorphism, the stack $\Min(\RLC)$ is naturally isomorphic to $\Mf^{\leq 1}/\Min(\LC)$.
\end{intheorem}
\begin{intheorem}[Theorem~\ref{thm:comp}]
\label{thm:mainarrlc}
For $r <\infty$ the stack $\Min(\RLC_r)/\Mf$ has a natural modular relative compactification $\Min(\ARLC_r)$ consisting of \emph{weakly $r$-rich log curves}. The stack $\Min(\ARLC_r)$ is a log blowup of the stack $\mathfrak{M}$ and the embedding $\Min(\RLC_r) \to \Min(\ARLC_r)$ is a map of log stacks.
\end{intheorem}
This gives $\Mf$ as an infinite union of quasi-compact stacks, with a compactification known for each step. Previously an explicit modular compactification of $\widetilde{\Mf}^{\leq r}$ was only known for $r = 1$.

\section{Conventions and notation}
There are multiple notions of log curves and tropical curves, and log stacks. In this section we explain our conventions, and we introduce our notation for the stacks of (minimal) (weakly) ($r$-) rich curves we will be defining in this paper. For a general reference on log schemes, we refer to \cite{ogus}.

For us, a log scheme $X$ consists of a scheme $\ul{X}$ and a sheaf of monoids $M_X$ on its small \'etale site together with a monoid map $\alpha: M_X \to \Ocal_{\ul X}$ such that the natural map $\alpha^* \Ocal_{\ul X}^\times \to \Ocal_{\ul X}^\times$ is an isomorphism. As usual, we denote the characteristic sheaf $M_X/\Ocal_X^\times$ by $\Mbar_X$. 

All log structures/log schemes will be fine and saturated (fs), meaning the sheaf of monoids is locally modelled on fine (finitely generated and integral), saturated monoids. In particular, all of our log blowups are saturated.

Of special interest to us will be log curves $C/S$. For a general reference, see \cite{kato2000}. A \emph{log curve} is a map $\pi: C \to S$ of log schemes such that the underlying algebraic map is a prestable curve and $\pi$ is log smooth and integral, and $C/S$ is vertical (no markings).
In the case where $S$ is a geometric point, the log structure becomes very explicit on the level of characteristic monoids, as follows. Letting $E$ denote the set of nodes (or equivalently, the set of edges of the dual graph), the log structure on $S$ will be a monoid $M$ over $k$ together with a monoid map $\N^E \to M$. Then the log structure on $C$ is uniquely determined by these data, as in \cite[Table~1.8]{kato2000}. We often call the image of $e \in \N^E$ in the monoid $\Mbar$ the \emph{length} or \emph{weight} of edge $e$.

Related is that for a prestable curve $\ul{C}/\ul{S}$, there is a unique initial log curve structure we can put on it (in the category of log curve structures on $\ul{C}/\ul{S}$). In the case where $\ul{S}$ is a geometric point, then this $C/S$ will satisfy that $\Mbar_S = \N^E$ where $E$ is the set of edges of the dual graph of $\ul{C}$. The length of the $e$th edge will then be the $e$th basis vector.

Finally, we discuss what tropical curves are for us. Fix a log smooth curve $C/S$, and let $s$ be a strict geometric point of $S$. Then $C_{\bar{s}}/\bar{s}$ has a dual graph $(V,E)$, a characteristic monoid $\Mbar_S$, and a metrisation, namely a length function $E \to \Mbar_S$ (with the image of an edge $e$ also called the length of an edge). This triple of data $((V,E),\Mbar_S,E \to \Mbar_S)$ is known as an \emph{abstract tropical curve}, and it is the tropical curve of $C$ at $s$, also denoted by $\Cf_s$. Furthermore, if we have an \'etale specialisation between two geometric points $s$ and $t$ (see \cite[Appendix~A]{cavalieri2020}), roughly meaning that $t$ lies in the closure of $s$, then we get contraction maps $\Cf_t \rightsquigarrow \Cf_s$, contracting some of the edges. The data of all these tropical curves at geometric points together with the contractions is known as the tropical curve of $C$, also denoted as $\Cf$. For a further reference to this, see Definition~2.3.3.3 of \cite{molcho18}.

\begin{remark}
A more common definition of an abstract tropical curve is a graph metrised in $\R_{\geq 0}$ (for an example, see Definition~2.13 of \cite{chan16}). One can see the more general definition of metrising in some sharp monoid $\Mbar$ as giving a family of $\R_{\geq 0}$-metrised graphs, parametrised by $\Hom(\Mbar,\R_{\geq 0})$. This perspective is worked out in Section~\ref{sec:trop}.
\end{remark}{}

\subsection{Log stacks}
\label{subsec:gillam}
In this paper we deal with stacks over schemes and stacks over log schemes. For an in depth treatment of the relation, we refer to \cite{gillam2011}. In this subsection, we treat a few of their definitions and results.
We denote by $\LogSch$ the category of all log schemes, and by $\Log$ the wide subcategory of log schemes with strict maps. This latter category is in fact a CFG (category fibred in groupoids) over $\Sch$.
\begin{definition}
A log CFG is a CFG $\ul{X}/\Sch$ together with a CFG map $M: \ul{X} \to \Log$ over $\Sch$. Such a log CFG $X = (\ul{X},M)$ induces a CFG $Y$ over $\LogSch$ by having as objects pairs $(x \in \ul{X},f: T \to M(x))$ where $\ul{f} : \ul{T} \to \ul{M(x)}$ is the identity on $\ul{T}$. We say that $X$, together with its log structure, represents $Y$, and denote this by $X = \Min(Y)$.
\end{definition}
\begin{example}
If we consider the CFG $\Sch_{\ul{S}}$ of schemes over $\ul{S}$ for some scheme $\ul{S}$, then a log structure $M: \Sch_{\ul{S}} \to \Log$ consists of lifting $\ul{S}$ to a log scheme $S$. The induced CFG over $\LogSch$ is then simply $\LogSch_{S}$, and $\Min(\LogSch_S)$ is the stack $\Sch_{\ul{S}}$ together with its log structure.
\end{example}

\begin{definition}
\label{def:gillam:log}
Let $X$ be a CFG over $\LogSch$. Then we denote by $\Log(X)$ the CFG over $\Sch$ consisting of the same objects as $X$, but only maps lying over a strict map.
\end{definition}
\begin{definition}
\label{def:logopen}
Let $f: X \to Y$ be a map in $\CFG/\LogSch$, with $X,Y$ representable by log CFG's. Then $f$ is an \emph{open immersion}, if the resulting map $\Min(X) \to \Min(Y)$ is an open immersion. It is \emph{log open} if the resulting map $\Log(X) \to \Log(Y)$ is an open embedding.
\end{definition}
\begin{remark}
This latter definition mimics many propositions about other log properties. For example, a map $X \to Y$ of log schemes is log \'etale if and only if $\Log(X) \to \Log(Y)$ is \'etale (\cite[Theorem~4.6]{olsson}).
\end{remark}

We will define two CFGs $\RLC$ and $\ARLC$ and a functor $f: \RLC \to \ARLC$ between them. We will then want to show that
\begin{enumerate}
  \item both $\RLC$ and $\ARLC$ are represented by log CFGs, and
  \item the map on log CFGs induced by $f$ is an open immersion.
\end{enumerate}
For this, the notion of ``minimal objects'' of a CFG over $\LogSch$, as developed by Gillam in \cite{gillam2011} will play an important role. This notion will allow us to explicitly describe the representing log CFGs.

\begin{definition}[\cite{gillam2011}]
\label{def:min} Let $F: X \to \LogSch$ be a CFG. Then an object $z \in X$ is \emph{minimal} if for every diagram
\[\begin{tikzcd}
w_1 &        & z \\
  & w_2 \arrow["i",ru] \arrow["j",lu] &  \\
\end{tikzcd}\]
with $\ul{Fi}$ and $\ul{Fj}$ being the identity on the underlying schemes, there is a unique map $w_1 \to z$ making the diagram commute.

\end{definition}
\begin{example}
Let $S$ be a log scheme. In the category $\LogSch_S$, the minimal objects are the strict maps $T \to S$, and hence correspond to maps of schemes $\ul{T} \to \ul{S}$. More generally, if our CFG $C/\LogSch$ is represented by a log CFG $(\ul{C},M)$, then the minimal objects are pairs $(c \in \ul{C},f: X \to M(c))$ where $f$ is an isomorphism.
\end{example}

\begin{theorem}[\cite{gillam2011}]
\label{thm:gillam}
Let $F:X \to \LogSch$ be a CFG. It is represented by a log CFG if and only it satisfies both of the following two conditions.
\begin{enumerate}
	\item For every $w \in C$, there is a minimal object $z$ and a map $i: w \to z$ with $\ul{Fi}$ the identity on the underlying schemes.
	\item For every minimal object $z$, for every map $i : w \to z$, we have that $w$ is minimal if and only if $Fi$ is strict. 
\end{enumerate}
Furthermore, if this is the case $C$ is represented by the stack $\Min(C)$ of minimal objects in $C$, with its natural log structure.
\end{theorem}

We often consider the log stack $\LC$ of log smooth curves, and the stack of prestable curves $\Mf$. We can give the latter a log structure coming from the boundary divisor of singular curves. Then we have the following theorem.
\begin{theorem}[Theorem~4.5,\hspace{2pt}\cite{kato2000}]
The stack $\LC/\LogSch$ is representable, and the stack with $\Mf$ with its natural log structure is naturally identified with $\Min(\LC)$.
\end{theorem}

When the base is unspecified, we work over $\Spec \Z$. The reader who prefers to work over $\C$ instead, can base change everything to $\Spec \C$ where all of our results will still hold.

\section{A tropical story}
\label{sec:trop}
In this section we give an overview of the main ideas of this paper written solely in the language of real tropical curves. We start with recalling some definitions about real tropical curves.

\begin{definition}
A \emph{real tropical curve} is a graph $\Gamma = (V,E)$ together with a length function $E \to \R_{>0}$. We denote by $M_\Gamma = \R_{> 0}^E$ the moduli space of real tropical curve structures on a fixed graph $\Gamma$.
\end{definition}
These moduli spaces fit together in families, as we will shortly see. Note that for a contraction of graphs $\Gamma \rightsquigarrow \Gamma'$, we have an inclusion of edges $E(\Gamma') \to E(\Gamma)$ and hence we can consider both $M_{\Gamma'}$ and $M_\Gamma$ as subspaces of $\R_{\geq 0}^E$.
\begin{definition}
Fix $\Gamma$ a graph. Denote $\Mbar_\Gamma = \R_{\geq 0}^{E(\Gamma)}$.
\end{definition}
\begin{proposition}
We have that $\Mbar_\Gamma = \bigsqcup_{\Gamma \rightsquigarrow \Gamma'} M_{\Gamma'}$ as sets.
\end{proposition}
\begin{example}
Let $\Gamma$ be the $2$-gon, i.e. the graph with two vertices and two edges connecting both vertices. Figure~\ref{fig:trop} contains a picture of $\Mbar_\Gamma$ for $\Gamma$ a $2$-gon, including four real tropical curves corresponding to points in $\Mbar_\Gamma$.

\begin{figure}
  \centering

  \begin{tikzpicture}[thick,scale=0.5]
    
    \fill [gray!20] (0,0) rectangle (10,10);
    \draw (0,0) -- (10,0);
    \draw (0,0) -- (0,10);

    \begin{scope}[every node/.style={circle, draw,fill=black,inner sep=0pt, minimum width=4pt}]
    
    \node (a) at (6,7) {};
    \node (b) at (8,7) {};
    \draw (a) to[out = 50, in = 130] (b);
    \draw (a) to[out = -50, in = -130] (b);

    \node (c) at (-1,6.7) {};
    \draw (-1,7.2) circle (0.5);

    \node (d) at (7,-0.5) {};
    \draw (7,-1) circle (0.5);

    \node (e) at (-1,-1) {};

    \end{scope}

  \end{tikzpicture}
  \caption{The moduli space $\Mbar_\Gamma$}
\label{fig:trop}

\end{figure}
\end{example}

\begin{remark}
If one wants to consider the moduli space of real tropical curves, one needs to take in account isomorphisms of real tropical curves. For example, in Figure~\ref{fig:trop}, curves on opposite sides of the diagonal are isomorphic as real tropical curves. One could consider quotienting $\Mbar_\Gamma$ by the isomorphism group $\Aut(\Gamma)$ of $\Gamma$, resulting in some different cone $\Mbar_\Gamma/\Aut(\Gamma)$. This turns out to be too coarse. The right thing to do is to consider the quotient \emph{cone stack}. This perspective is treated in detail in \cite{cavalieri2020}. In this section, instead of working over this cone stack, we will work with structures over $\Mbar_\Gamma$ invariant under the natural $\Aut(\Gamma)$-action, which amounts to the same thing.
\end{remark}

\begin{definition}
A family of real tropical curves over a cone $\sigma \subset \R^n$ consists of a graph $\Gamma$ and a linear map $\sigma \to \Mbar_\Gamma$. 
\end{definition}
\begin{example}
Let $C \in \Mbar_\Gamma$ be a real tropical curve. Then the map $\R_{\geq 0} \to \Mbar_\Gamma$ mapping $\lambda$ to $C$ with all edge lengths scaled by $\lambda$ is a family of real tropical curves.
\end{example}
\begin{example}
There is a universal family of real tropical curves $C_\Gamma$ over the moduli space $\Mbar_\Gamma$ itself.
\end{example}
\begin{example}
Let $C/S$ be a log curve with $S$ a geometric point with characteristic monoid $\Mbar$. Its corresponding tropical curve is a graph $\Gamma$ together with edge lengths in $\Mbar$. Let $\sigma$ be the cone $\Hom(\Mbar,\R_{\geq 0})$. Then evaluation of the edge lengths defines a family of real tropical curves $\sigma \to \Mbar_\Gamma$.
\end{example}

We can now define, just like in the introduction, when such a family is weakly rich.
\begin{definition}
Given a connected graph $\Gamma = (V, E)$, a \emph{cut} of $\Gamma$ is a subset of $E$ such that $V \sqcup (E \setminus c)$ has exactly two connected components, and every edge in $c$ connects these two components. We denote the set of all cuts by $\Cuts(\Gamma)$.
\end{definition}

\begin{definition}
Let $\ell: \sigma \to \Mbar_\Gamma$ be a family of real tropical curves. Let $E = E(\Gamma)$ and for $e \in E$ denote $\ell(e): \sigma \to \R_{\geq 0}$ the length of the $e$th edge. Then we say this family is \emph{weakly rich} if for every cut $c \subset E$ of $\Gamma$, there is an $e \in c$ such that $e$ is universally the smallest edge in $c$, meaning that for all $x \in \sigma$ and for all $e' \in c$ we have $\ell(e)(x) \leq \ell(e')(x)$.
\end{definition}
\begin{example}
\label{ex:lineweaklyrich}
If $\sigma$ is $\R_{\geq 0}$, the corresponding real tropical curve is weakly rich.
\end{example}
\begin{example}
Let $\Gamma$ be a graph. The universal family of real tropical curves $C_\Gamma$ is weakly rich if and only if $\Gamma$ is a tree with loops added.
\end{example}
\begin{example}
\label{ex:trop2gon}
Let $\Gamma = (V,E)$ be the $2$-gon. Let $\sigma \subset \R^2$ be the cone generated by $(1,0)$ and $(1,1)$. Choose a bijection $E \to \{1,2\}$ and let $\sigma \to \R^E = \Mbar_\Gamma$ be the natural inclusion for this bijection. This is an weakly rich curve. It is even a maximal weakly rich subcurve of the universal curve $\Mbar_\Gamma \to \Mbar_\Gamma$. There is one other maximal weakly rich subcurve, obtained by choosing the other bijection $E \to \{1,2\}$.
\end{example}

Blowups on the logarithmic side of the story correspond exactly to subdivisions of $\Mbar_\Gamma$ on the tropical side. In Section~\ref{sec:compact} we will see that the space of weakly rich log curves is a blowup of the space of rich log curves. As a consequence, there is a subdivision of $\Mbar_{\Gamma}$ such that a family of real tropical curves $\sigma \to \Mbar_{\Gamma}$ is weakly rich if and only if it factors through the subdivision. It turns out one can explicitly write down this subdivision, and prove that it is in fact a subdivision.
\begin{definition}
Let $\Gamma$ be a graph. Let $f: \Cuts(\Gamma) \to E$ be a choice function, sending a cut $c$ to an element of $c$. Then the corresponding cone $\Mbar_f \subset \Mbar_\Gamma$ is defined to consist of those real tropical curves such that for every cut $c$ we have that $f(c)$ is the shortest edge in that vertex. We define the weakly rich subdivision of $\Mbar_\Gamma$ to be the set of faces of these cones for every possible choice function.
\end{definition}

\begin{figure}[ht]
\begin{subfigure}{.49\textwidth}
  \centering

  \begin{tikzpicture}[thick,scale=0.5]
    
    \fill [gray!20] (0,0) rectangle (10,10);
    \draw (0,0) -- (10,0);
    \draw (0,0) -- (0,10);
    \draw (0,0) -- (10,10);

  \end{tikzpicture}
  \caption{A subdivision of the moduli space $\Mbar_\Gamma$ with $\Gamma$ a $2$-gon}
  \label{fig:tropsub1}
\end{subfigure}
\begin{subfigure}{.49\textwidth}
  \centering
  \begin{tikzpicture}[thick,scale=0.56]
    
    \begin{scope}[every node/.style={draw=none, inner sep =-0pt}]
    \node (B) at (90:6) {};
    \node (C) at (210:6) {};
    \node (D) at (330:6) {};

    \def\x{1};
    
    \draw [shorten <= -\x pt, shorten >= -\x pt] (B) -- (C) node (BC) [midway] {} node (BBC) [pos = .333] {} node (BCC) [pos = .666] {};
    \draw [shorten <= -\x pt, shorten >= -\x pt] (B) -- (D) node (BD) [midway] {} node (BBD) [pos = .333] {} node (BDD) [pos = .666] {};
    \draw [shorten <= -\x pt, shorten >= -\x pt] (C) -- (D) node (CD) [midway] {} node (CCD) [pos = .333] {} node (CDD) [pos = .666] {};

    \node (1) at (0:0) {};
    \node (2) at (intersection of  B--CD and C--BBD){};
    \node (3) at (intersection of  B--CD and C--BDD){};
    \node (4) at (intersection of  C--BD and B--CCD){};
    \node (5) at (intersection of  C--BD and B--CDD){};
    \node (6) at (intersection of  D--BC and B--CDD){};
    \node (7) at (intersection of  D--BC and B--CCD){};

    \draw [shorten <= -\x pt, shorten >= -\x pt] (B) -- (3);
    \draw [shorten <= -\x pt, shorten >= -\x pt] (B) -- (4);
    \draw [shorten <= -\x pt, shorten >= -\x pt] (B) -- (5);
    \draw [shorten <= -\x pt, shorten >= -\x pt] (B) -- (6);
    \draw [shorten <= -\x pt, shorten >= -\x pt] (B) -- (7);

    \draw [shorten <= -\x pt, shorten >= -\x pt] (C) -- (5);
    \draw [shorten <= -\x pt, shorten >= -\x pt] (C) -- (2);
    \draw [shorten <= -\x pt, shorten >= -\x pt] (C) -- (3);
    \draw [shorten <= -\x pt, shorten >= -\x pt] (C) -- (6);
    \draw [shorten <= -\x pt, shorten >= -\x pt] (C) -- (7);

    \draw [shorten <= -\x pt, shorten >= -\x pt] (D) -- (7);
    \draw [shorten <= -\x pt, shorten >= -\x pt] (D) -- (2);
    \draw [shorten <= -\x pt, shorten >= -\x pt] (D) -- (3);
    \draw [shorten <= -\x pt, shorten >= -\x pt] (D) -- (4);
    \draw [shorten <= -\x pt, shorten >= -\x pt] (D) -- (5);

    \end{scope}

  \end{tikzpicture}\caption{A subdivision of $\R_{\geq 0}^3$, intersected with the plane $a + b + c = 1$}
  \label{fig:tropsub2}
\end{subfigure}
\caption{}
\label{fig:trop2}
\end{figure}

\begin{theorem}
\label{thm:trop:subdivision}
The weakly rich subdivision of $\Mbar_\Gamma$ is a subdivision, and a family of real tropical curves $\sigma \to \Mbar_\Gamma$ is weakly rich if and only if it factors through the weakly rich subdivision.
\end{theorem}

We remark that this theorem also follows for free from Corollary~\ref{cor:blowup}. Here we sketch a purely tropical proof of this statement.

\begin{proof}[Proof of Theorem~\ref{thm:trop:subdivision}]
Consider the function 
\[
g: \Mbar_{\Gamma} \to \R_{\geq 0}^{\Cuts(\Gamma)}, (x_e)_{e \in E} \mapsto (\min_{e \in c} \{x_e\})_{c \in \Cuts(\Gamma)}.
\]
Let $S \subset \Mbar_{\Gamma}$ be the locus where one of the minima is achieved twice, i.e. the locus where $g$ is not linear. This locus is a polyhedral fan, and hence defines a subdivision of $\Mbar_\Gamma$. We see the maximal cones in this subdivision are of the form $\Mbar_f$ for some choice function $f$, and conversely, any $\Mbar_f$ is a cone in this subdivision. So the subdivision of $\Mbar_\Gamma$ induced by $S$ is the weakly rich subdivision, and hence the weakly rich subdivision is indeed a subdivision.

For the second part, first note that the cones $\Mbar_f$ in the weakly rich subdivision are weakly rich. So any family of real tropical curves factoring through one of the $\Mbar_f$ is weakly rich. Conversely, if $\sigma \to \Mbar_\Gamma$ is weakly rich, then there is a choice function $f: \Cuts(\Gamma) \to E$ sending every cut to a smallest edge contained in that cut. Then $\sigma \to \Mbar_\Gamma$ will factor through $\Mbar_f$. 
\end{proof}

\begin{example}
For the $2$-gon, we get the subdivision described in Example~\ref{ex:trop2gon} and shown in Figure~\ref{fig:tropsub1}. For the $3$-gon, we get the subdivision from Figure~\ref{fig:sub-third} as described in Example~\ref{ex:intro} and in more detail in Example~\ref{ex:subdiv}.
\end{example}

\begin{remark}
In this section, we only investigated the case for weakly rich curves, which will be known as weakly $r$-rich curves for $r=1$. For larger $r$, one still gets a corresponding tropical story, and subdivisions of the tropical moduli spaces. For $\Gamma$ the graph with two vertices and three edges and for $r = 2$ we give the corresponding subdivision in Figure~\ref{fig:tropsub2}.
\end{remark}

Many logarithmic theorems in this paper have an analogue or shadow in this tropical world. For example, the fact that the stack $\Min(\ARLC)$ is proper over $\Mf$ is equivalent to the fact that the almost rich subdivision is a complete subdivision of $\Mbar_{\Gamma}$ for every $\Gamma$. Similarly, the fact that $\Min(\ARLC)$ is smooth, which will be proven in Theorem~\ref{thm:smooth}, is equivalent to the fact that every maximal dimensional cone has a $\R_{\geq 0}$-basis that is also a $\Z$-basis of $\Z^{E(\Gamma)}$.

\section{The stack of \texorpdfstring{$r$-rich}{r-rich} log curves and the stack of \texorpdfstring{$r$-enriched}{r-enriched} curves}
\label{sec:rich}
In this section we will define $r$-richness of log curves with $r \in \N_{\geq 1} \cup \{\infty\}$, and study the moduli spaces of $r$-rich curves. For $r = 1$ we get the simpler notion of richness, and throughout this and the following sections, one can always keep the case of $r = 1$ in mind as a representative case.

\begin{definition}
Let $M$ be an fs monoid and $A = \{a_1,\dots,a_n\}$ a subset of $M$. Let $r \in \N_{\geq 1} \cup \{\infty\}$. Then the set $A$ is called \emph{$r$-close} if there are positive integers $\lambda_i \mid r$ and $a \in M$ such that for all $i$ we have $\lambda_i a = a_i$ (where $\infty$ is divisible by any positive integer). For an $\infty$-close set, we call the largest such $a$ dividing all elements of $A$ the \emph{root} of $A$.
\end{definition}

Note that being $\infty$-close just means lying on the same ray $\N \to M$ sending $1$ to a non-divisible element. That immediately gives us the following lemma.
\begin{lemma}
For $M$ an fs monoid and $A$ and $A'$ two non-disjoint $\infty$-close subsets, the union $A \cup A'$ is $\infty$-close.
\end{lemma}
\begin{remark}
\label{rem:61015}
Note that if $A$ is a set such that all subsets of size $2$ are $r$-close, then $A$ is $r$-close itself.
\end{remark}

Throughout the rest of this section, $r$ will lie in $\N_{\geq 1} \cup \{\infty\}$. 

\begin{definition}
Let $C/\Spec k$ be a log curve over an algebraically closed field and let $\Gamma = (V, E)$ be the corresponding tropical curve. The curve $C/\Spec k$ is called \emph{$r$-rich} if for every cut $c$ of $\Gamma$, the set of lengths of the edges in $c$ are $r$-close. A log curve $C/S$ is called \emph{$r$-rich} if it is $r$-rich over all strict geometric points. When $r = 1$, we also call this \emph{rich}.
\end{definition} 

\begin{example}
A curve $C/\Spec k$ with $k = \overline{k}$ is rich if and only if for each cut of its dual graph $\Gamma$, all edges in the cut have equal lengths in the characteristic monoid.
\end{example}

By Remark~\ref{rem:61015} we have the following, more concrete reformulation, in terms of so-called circuit-connected components.
\begin{definition}
\label{def:graph}
Let $G = (V, E)$ be a graph. A \emph{path} is a sequence of directed edges $e_1,\dots,e_n$ with $e_i$ ending where $e_{i+1}$ starts. It is a \emph{circuit} if furthermore $e_n$ ends where $e_1$ starts and every vertex in $V$ is incident to at most two of the edges. We call a subset of $E$ \emph{circuit-connected} if for any two distinct edges $e,e'$, there is a circuit containing $e,e'$ . This partitions $E$ into a collection $T$ of maximal \emph{circuit-connected components} \cite[Lemma 7.2]{holmes2015}.
\end{definition}
\begin{lemma}
\label{lem:61015}
Let $C/\Spec k$ be a log curve over an algebraically closed field and let $\Gamma = (V, E)$ be the corresponding tropical curve. Then $C/\Spec k$ is $r$-rich if and only if for every circuit-connected component, the labels of the edges it contains are $r$-close.
\end{lemma}
\begin{proof}
This follow from combining Remark~\ref{rem:61015} together with the fact that any cut is contained in a circuit-connected component, and the fact that any size two subset of a circuit-connected component can be extended to a cut.
\end{proof}

We also have the following alternative definition in terms of the existence of piecewise linear functions.

\begin{proposition}
\label{lem-eq-defs-richness}
Let $C/\Spec k$ be a log curve over an algebraically closed field and let $\Gamma = (V, E)$ be the corresponding tropical curve. Then $C$ is $r$-rich if and only if for every cut of $\Gamma$, there exists a piecewise linear function on $\Gamma$ such that the slope on an edge in $c$ divides $r$ and the slope on all other edges is $0$.
\end{proposition}

\begin{example}
Consider the $2$-gon over a point $\Spec k$: this is the genus $1$ curve obtained by taking two copies of $\P^1$, glueing the two copies of $(0 : 1)$ together, and glueing the two copies of $(1: 0)$ together. The minimal log structure on it is not $1$-rich, as the characteristic monoid of the minimal log structure is $\N^2$ on the base with the two edges having labels $e_1$ and $e_2$ respectively. A $1$-rich log structure is then up to isomorphism a map $k^\times \oplus \N^2 \to k^\times \oplus \Mbar$ which on the level of characteristic monoids sends $e_1$ and $e_2$ to the same element in $\Mbar$, where $\Mbar$ is an arbitrary sharp monoid. Such a map always factorises through an isomorphism $k^\times \oplus \N^2 \to k^\times \oplus \N$ which on the level of characteristic monoids sends $e_1$ and $e_2$ to $1 \in \N$. The set of such possible maps $k^\times \oplus \N^2 \to k^\times \oplus \N$ up to isomorphism will turn out to be in bijection with $k^\times$. This example will be generalised and explained in more detail in Example~\ref{ex:richalgclosed}.
\end{example}

\begin{remark}
The notion of $r$-richness is connected to, but slightly different both from the notion of being $r$-aligned from \cite{holmes2015} and from the notion of being log aligned from \cite{poiret2020}. The difference from being $r$-aligned, except for the important fact that being $r$-aligned is a condition of a scheme instead of a log scheme, is mainly that for $r$-aligned one requires the labels in a circuit-connected component to all be of the form $\lambda_i a$ for some $1 \leq \lambda_i \leq r$, instead of requiring the labels in a cut to be of the form $\mu_i a$ for some $1 \mid \mu_i \mid r$. The reason for cuts is that, while circuit-connected components can split when contracting edges, cuts do behave nicely under contraction (cf. Lemma~\ref{lem:graphcuts}).

The difference from being log aligned is that then one requires the labels of the edges in a circuit-connected component to lie on an extreme ray.

Both of these alternative notions make sense in their own right, but both are not closed under pullback and hence one cannot form an algebraic or log stack by taking the category fibred over groupoids of these objects; one has to work with the terminal object in the right category instead.
\end{remark}

With a notion of $r$-richness, we can now define the category of objects we are interested in.

\begin{definition}
\label{def:rlc}
Define $\RLC_r \to \LogSch$ to be the sub-CFG consisting of $r$-\textbf{R}ich \textbf{L}og \textbf{C}urves.
\end{definition}

We have maps between these spaces, as described in the following lemma.
\begin{lemma}
\label{lem:covariancer}
Let $r,r' \in \N_{\geq 1} \cup \{\infty\}$ with $r \mid r'$. Any $r$-rich log curve $C/S$ is $r'$-rich. There is an inclusion
\begin{equation}
\RLC_r \subset \RLC_{r'} \subset \LC.
\end{equation}
\end{lemma}
\begin{proof}
Note that any $r$-close set is $r'$-close for any $r'$ divisible by $r$, and in particular $\infty$-close. The statement follows immediately.
\end{proof}

In this section, we will prove the following theorem about the stack $\RLC_r$. 

\begin{theorem}
\label{thm:rich}
The stack $\RLC_r/\LogSch$ is a log \'etale substack of the stack of log curves $\LC$, and is represented by a stack with log structure $\Min(\RLC_r)/\Min(\LC)$.
\end{theorem}

In the next subsection, we will prove this using Gillam's framework of minimality, as discussed in Section~\ref{subsec:gillam}. We will state what the minimal objects are, and that the conditions for Theorem~\ref{thm:gillam} are satisfied and hence $\RLC_r$ is represented by the log CFG $\Min(\RLC_r)/\Sch$. The characterisation of minimal objects is novel but the proofs are technical and standard, and are given in Appendix~\ref{sec:appendix}. Next we will show that $\Log(\RLC_r)$ (for a definition of $\Log(\cdot)$ see Definition~\ref{def:gillam:log}) is an open substack of $\Log(\LC)$, and finally we will be able to conclude Theorem~\ref{thm:rich}.

\subsection{\texorpdfstring{Properties of $\RLC_r$}{Properties of RLCr}}
\begin{proposition}[Proposition~\ref{app:prop:richmineqbasic}, Definition~\ref{app:def:richbasic}]
\label{def:richminimal}
Let $C/\Spec k$ be a $r$-rich log curve over an algebraically closed field and let $\Gamma = (V, E)$ be the corresponding tropical curve. Let $T$ denote the set of circuit-connected components, for $t \in T$ let $a_t \in \Mbar_k$ be the root of $T$, and for $e \in t$ let $\lambda_e \mid r$ be such that the length of the $e$th edge is $\lambda_e a_t$. Then the natural map $\N^E \to \Mbar_k$ factors through the \emph{root map} $\N^E \to \N^T$ where $T$ is the set of circuit-connected components and $\N^E \to \N^T$ sends $e$ to $\lambda_e$, and $\N^T \to \Mbar_k$ sends $t$ to $a_t$. The log curve $C/\Spec k$ is minimal if the resulting map $\N^T \to \Mbar_k$ is an isomorphism. An $r$-rich log curve $C/S$ is minimal if and only if it is basic $r$-rich over all strict geometric points.
\end{proposition}

\begin{proposition}[Proposition~\ref{app:prop:mrlc}]
\label{prop:mrlc}
The category $\RLC_r$ satisfies the conditions of the descent lemma (Theorem~\ref{thm:gillam}).
\end{proposition}

Then by Theorem~\ref{thm:gillam} we get the following proposition.

\begin{proposition}
\label{def:mrlc}
The subcategory $\Min(\RLC_r)$ of minimal objects of $\RLC_r$ is a stack over $\Sch$ and together with its natural log structure represents $\RLC_r$.
\end{proposition}

There is the natural inclusion of CFGs $\RLC_r \to \LC$. We will later see in Section~\ref{sec:compact} that for $r < \infty$ this is an open embedding (cf.  Definition~\ref{def:logopen}) followed by a log blow up (and hence a log \'etale monomorphism), but it is not an open embedding itself. The issue is that minimal $r$-rich log curves are not necessarily minimal log curves. There is always a map from a minimal $r$-rich log curve to a minimal log curve lying over the identity of the underlying algebraic scheme, but this map is not necessarily strict. The following proposition shows that this non-strictness is the only issue here.

\begin{proposition}
\label{prop:richopen}
Let $\pi : C \to S$ be a log curve, and let $U$ be the full subcategory of objects $T$ in $\Sch/\ul{S}$ such that the strict pullback $C_T$ is $r$-rich. Then $U \to S$ is an open embedding.
\end{proposition}
\begin{proof}
We will prove this proposition \'etale locally. By \cite[Lemma~3.40]{poiret2020} we can assume our curve $C/S$ is nuclear as defined in \cite[Definition~3.39]{poiret2020}. Now let $\overline{s}$ be a strict geometric point where the scheme is rich, and shrink $S$ such that every tropical curve of a point of $S$ is a contraction of the tropical curve at $\overline{s}$. Let $E$ then be the edge set of the tropical curve $\Cf_s$. Then we can describe the entire tropical curve $\Cf/S$ by giving for each geometric point $\overline{t}$ of $S$ a map $E \to \Mbar_{\overline{t}}$, with the image of an edge $e$ vanishing exactly if $e$ gets contracted over $\overline{t}$. Let $T$ be the set of circuit connected components of $E$, and let $q : \N^E \to \N^T$ be the root map at $\overline{s}$, as defined in Proposition~\ref{def:richminimal}. For $e \in E$, let $l_e$ denote the length of edge $e$ as a global section of the characteristic monoid $\Mbar_S$. For all pairs $e,e' \in E$ in the same circuit-connected component there are integers $\lambda_e,\lambda_{e'}$ with $\lambda_{e'}q(e) = \lambda_e q(e')$. Then we have $\lambda_{e'}l_{e} = \lambda_e l_{e'}$ in the stalk $\Mbar_{\overline{s}}$, hence this equality also holds in some open neighbourhood $V_{e,e'}$ of $\overline{s}$. Then we see that $\bigcap_{t \in T} \bigcap_{e,e' \in t} V_{e,e'}$ is an open neighbourhood of $\overline{s}$ that is contained in the $r$-rich locus. Hence, the $r$-rich locus $U$ is open.

Now by definition of $U$, the open $U$ has the universal property that for any strict map $f: T \to S$, the pullback $C_T$ is $r$-rich if and only if $f$ factors through $U$.
\end{proof}

\begin{corollary}
The map $\RLC_r \to \LC$ is a log \'etale monomorphism.
\end{corollary}
\begin{proof}
Note that for a map of stacks $f: X \to Y$ over $\LogSch$, we have that $f$ is log \'etale if and only if $\Log(X) \to \Log(Y)$ is \'etale by \cite[Theorem~4.6]{olsson}. By Proposition~\ref{prop:richopen}, we then have that $\RLC_r \to \LC$ is log open, hence it is log \'etale. Also, it is a sub-CFG, which completes the proof.
\end{proof}

This finishes the proof of Theorem~\ref{thm:rich}.


\begin{remark}
At this point we could prove that $\Min(\RLC_r)$ is a smooth algebraic stack by applying \cite[Theorem~B.2]{wise2016}. However, we postpone this to Section~\ref{subsec:rrichagain}, when we will have shown that $\ARLC_r$ is a log blowup of $\LC$ and $\Min(\RLC_r)$ is an open inside $\Min(\ARLC_r)$, as the algebraicity of $\Min(\RLC_r)$ will then follow for free.
\end{remark}

\begin{example}
Let $C/\overline{s}$ be a compact type prestable curve over a geometric point. Then by definition the dual graph $(V,E)$ is a tree. Then the set of circuit-connected components is $E$ itself, and one can check the minimal log structure is in fact the minimal $r$-rich log structure. Letting $X$ denote the stack of compact type curves, it follows that the map $\Min(\RLC_r) \times_{\Min(\LC)} X \to X$ is an isomorphism. 
\end{example}

\begin{example}
\label{ex:richalgclosed}
Let $C/S$ be a prestable curve over a geometric point $S = \Spec k$. Let $E$ be the edge set of the corresponding dual graph, and let $T$ be the set of circuit-connected components. A $1$-rich log structure $(M_C,M_S)$ on $C$ is completely specified by a map of log structures from the minimal log structure $\N^E \oplus k^*$ on the base to $M_S$. As we are working over an algebraically closed field, that would just be a map of monoids $\N^E \oplus k^* \to \N^T \oplus k^*$ that, on the level of characteristic monoids, is the linearisation of the quotient map $c : E \to T$.

These maps are specified by giving $v \in (k^*)^E$, and denote by $\phi_v$ the map sending $(e,1) \in \N^E \oplus k^*$ to $(c(e),v_e)$. An automorphism of the log structure $\N^T \oplus k^*$ is given by $w \in (k^*)^T$, by for $t \in T$ sending $(t,1) \in \N^T \oplus k^*$ to $(t,w_t)$. These automorphisms induce isomorphisms $\phi_v \to \phi_{v\cdot c^*(w)}$. Then the minimal rich log structures correspond exactly to the maps $\phi_v$ up to isomorphism.

Three examples: for a $1$-gon (a graph with one vertex and one edge), all choices of $\phi_v$ are isomorphic, and there is only one minimal rich log structure (up to unique isomorphism). For a $n$-gon in general, we get $(k^*)^n/k^*$ isomorphism classes, where the action of $k^*$ is coordinate-wise multiplication. In full generality, these explicit constructions give a natural bijection between the minimal rich log structures on $C/S$ and $\G_m^E/\G_m^T$.

For $r > 1$, there is more choice. For ease of notation, assume there is one circuit-connected component. Then a minimal $r$-rich log structure $M$ on $C/S$ has $\Mbar_k = \N$, with the length function $E \to \N$ sending every edge to a divisor $\lambda_e$ of $r$ such that $\gcd_{e \in E} \lambda_e = 1$. Fixing a choice of $(\lambda_e)_{e \in E}$, we get that the minimal rich log structures are naturally in bijection with $[\G_m^E/\G_m]$ where $\G_m$ acts with weight $\lambda_e$ on the $e$'th coordinate. As the $\lambda_e$ are coprime, this is isomorphic to $\G_m^{|E|-1}$.
\end{example}

\section{Comparison to work by Main\`o}
\label{sec:maino}
Although it will follow from Chapter~7 in \cite{biesel2019} and Section~\ref{sec:compact}, we will show here directly that our notion of basic $1$-rich log structures on a curve coincides with the notion of enriched structures on a curve when working over an algebraically closed field.
. We first give maps between the sets of basic $1$-rich log structures and the sets of enriched structures on a prestable curve. Throughout this section, we fix $C/\Spec k$ a prestable curve with $k$ algebraically closed, with dual graph $\Gamma = (V, E)$ and components $(C_v)_{v \in V}$, and we let $C_v^c$ be $\bigcup_{w \neq v} C_w$.
For enriched structures, we use Definition~\ref{def:richoverk} and the equivalent following definition paraphrased from \cite{maino1998}. 
\begin{proposition}[\protect{\cite[Proposition~2.14]{maino1998}}]
Let $C/\Spec k$ be as described above. Then a sequence of line bundles $(\Lcal_v)_{v \in V}$ is an enriched structure if and only if the following two conditions hold:
\begin{enumerate}
  \item for every vertex $v \in V$ we have $\Lcal_v|_{C_v} \cong \Ocal_{C_v}(-C_v \cap C_v^c)$ and $\Lcal_v|_{C_v^c} \cong \Ocal_{C_v}(C_v \cap C_v^c)$;
  \item the tensor product $\prod_{v \in V} \Lcal_v$ is isomorphic to $\Ocal_C$.
\end{enumerate}
\end{proposition}

Now we can introduce the enriched structure associated to a basic $1$-rich log structure.

\begin{definition}
\label{def:assocenriched}
Let $M = (M_C,M_k)$ be a basic $1$-rich log structure on $C/\Spec k$. For $v \in V$, let $\delta_{M,v}$ be a piecewise linear function on $\Gamma$ taking values in $M^\gp$ that has slopes $-1$ going out from $v$ and slopes $0$ everywhere else, and let $\Lcal_v = O_C(-\delta_{M,v})$. Define $M^{as}$ to be the family $(\Lcal_v)_{v \in V}$.
\end{definition}
\begin{proposition}
Let $M$ be a basic $1$-rich log structure on $C/\Spec k$. The family $M^{as}$ is an enriched structure on $C$.
\end{proposition}
\begin{proof}
Because $\sum_v \delta_{M,v}$ is constant, we have $\prod \Lcal_v \cong \Ocal_x$. Then by Proposition~2.4.1 of \cite{ranganathan2019}, these $\Lcal_v$ satisfy \[\Lcal_v|_{C_v} \cong \Ocal_{C_v}(-C_v \cap C_v^c)\] and \[\Lcal_v|_{C_v^c} \cong \Ocal_{C_v^c}(C_v \cap C_v^c)\] and hence form an enriched structure by \cite[Definition~2.14]{maino1998}.
\end{proof}

We will now describe the map the other way. Note that given a regular smoothing $\Ccal/\Spec k[[t]]$, we can put a canonical log curve structure on it. We do this by giving $\Spec k[[t]]$ the log structure by the divisor $t = 0$ and $\Ccal$ the log structure by the divisor $C$. This is in fact a $1$-rich log structure, and restricts to a $1$-rich log structure on $C/\Spec k$. Then by Proposition~\ref{prop:mrlc} this corresponds to a unique minimal $1$-rich log structure on $C/\Spec k$.

\begin{definition}
Let $(\Lcal_v)_v$ be an enriched structure on $C/\Spec k$. Let $\Ccal/\Spec k[[t]]$ be a regular smoothing such that $\Lcal_v \cong O_{\Ccal}(C_v)|_C$. Let $(M_C,M_k)$ be the basic $1$-rich log structure on $C/\Spec k$ induced by this regular smoothing, as described above. This is called the basic $1$-rich log structure associated to $(\Lcal_v)_v$. 
\end{definition}
A priori this depends on the choice of the regular smoothing, but we will show that these two procedures are actually inverses, which means the choice does not matter. To show that, we have the following two propositions.
\begin{proposition}
Let $\Ccal/S$ with $S = \Spec k[[t]]$ be a regular smoothing of $C$. Let $N_C$ be the associated basic $1$-rich log structure on $C$. Then the associated enriched structure $(O_C(-\delta_{N_C,v}))_{v \in V}$ is isomorphic to the enriched structure $(O_{\Ccal}(C_v)|_C)_{v \in V}$.
\end{proposition}
\begin{proof}
We need to show that for every $v \in V$, the line bundle $O_C(-\delta_{N_C,v})$ is isomorphic to $O_{\Ccal}(C_v)|_C$. We fix a $v \in V$. We introduce some other log structures: we let $M_\Ccal$ be the associated divisorial $1$-rich log structure on $\Ccal/S$ and we let $M_C$ be the restriction of this to $C$.

We recall the definition of $O_C(-\delta_{N_C,v})$: we start with the exact sequence
\[
0 \to O_C^\times \to N_C^\gp \to \Nbar_C^\gp \to 0.
\]
Then $\delta_{N,v}$ lives in $\Nbar_C^\gp$, hence the inverse image in $N_C^\gp$ is an $O_C^\times$-torsor, which corresponds to a line bundle on $O_C$ by adding the infinity section. This line bundle is then $O_C(-\delta_{N_C,v})$. Note that the image $\alpha$ of $\delta_{N,v}$ in $\Mbar_C^\gp$ induces the same line bundle as $\delta_{N,v}$. This section $\alpha$ is the restriction of a piecewise linear function $\beta \in \Mbar_\Ccal^\gp$. We have $O_\Ccal(-\beta)|_C = O_C(-\alpha) = O_C(-\delta_{M,v})$. Recall that $M_\Ccal$ is by definition the divisorial $1$-rich log structure associated to the divisor $C$. The corresponding log structure $M_S$ on $S$ is induced by the map $\N \to k[[t]]$ sending $1$ to $t$. Then $\beta$ is the piecewise linear function that is $1$ on $C_v$ and $0$ on $C_w$ with $w \neq v$, and trivial everywhere on $\Ccal \setminus C$ (where the characteristic monoid is trivial). Now we can compute $O_\Ccal(-\beta)$ by hand. The preimage of $-\beta$ in $M_\Ccal^\gp$ consists of functions in the function field invertible outside of $\Ccal$ that have a simple zero at $C_i$ and are units everywhere else, which after filling in the infinity section corresponds exactly to the line bundle $O_\Ccal(C_i)|_C$.
\end{proof}

\begin{proposition}
Let $M = (M_C,M_k),M' = (M_C',M_k')$ be two basic $1$-rich log structures on $C/\Spec k$ such that the enriched structures $(O_C(-\delta_{M,v}))_{v \in V}$ and $(O_C(-\delta_{M',v}))_{v \in V}$ are isomorphic. Then $M$ and $M'$ are isomorphic.
\end{proposition}
\begin{proof}
Let $T$ denote the set of circuit-connected components of $\Gamma$. Note that on the level of characteristic monoids, any two basic $1$-rich log structures are canonically isomorphic. Namely, $M$ and $M'$ both satisfy $\Mbar_k = \Mbar_k' = \N^T$, with the length of edge $e$ being the circuit-connected component it is contained in. By \cite[1.8]{kato2000} this uniquely determines $\Mbar_C$ and $\Mbar_C'$. Note that $M_k$ and $M_k'$ are also isomorphic, since they both split as $k^* \oplus \N^T$.

By the proof of \cite[Theorem~3.6]{borne2012}, we can associate Deligne-Faltings log structures to a log structure $M_X$ on a scheme $X$. This Deligne-Faltings log structure is a symmetric monoidal map $\Mbar_{X} \to \Div_{X_{\et}}$ where $\Div_{X_{\et}}$ parametrises line bundles with sections on the small \'etale site of $X$.

By \cite[Theorem~3.6]{borne2012}, the two log structures $M_C$ and $M'_C$ are isomorphic if and only if their associated Deligne-Faltings log structures are isomorphic.

For this, we have to show that for two piecewise linear functions $\alpha_M \in \Mbar_C$ and $\alpha_{M'} \in \Mbar_C'$ with $\alpha = \alpha'$ under the isomorphism $\Mbar_C \cong \Mbar_C'$ we have that the induced line bundle $O_C(\alpha_M), O_C(\alpha_{M'})$ together with their respective sections $s = 1,s' = 1$ are isomorphic.

As the two log structures are basic rich, any piecewise linear function is a combination of a constant piecewise linear function and of the $\delta_{M_v}$ for different vertices $v$. Fix a vertex $v$. We already know that the line bundles $O_C(-\delta_{M,v}),O_C(-\delta_{M',v})$ are isomorphic by assumption. It remains to check that the isomorphisms can be chosen to respect the sections $s, s'$.

By the description of $O_C(-\delta_{M,v})$ given in \cite[2.4.1]{ranganathan2019} we can identify $O_C(-\delta_{M,v})|_{C_v}$ with $O_{C_v}(C_v \cap C_v^c)$ and $O_C(-\delta_{M,v})|_{C_{v'}}$ with $O_{C_{v'}}(- C_v \cap C_{v'})$ for any other vertex $v'$. With the identifications made, the sections $s$ and $s'$ are $1$ on $C_v$ and $0$ on $C_v^c$. This means that the line bundles with sections $(O_C(-\delta_{M,v}),s)$ and $(O_C(-\delta_{M',v}),s')$ are canonically isomorphic. Then the Deligne--Faltings structures associated to $M$ and $M'$ are isomorphic, so as stated previously we get that $M$ and $M'$ are isomorphic.
\end{proof}

\begin{corollary}
\label{cor:maino}
Let $C/\Spec k$ be a prestable curve over an algebraically closed field, and let $k \to \Min(\LC)$ be the corresponding map. Then the set $\Min(RLC)_{\Spec k}(k)$ of $k$-points of the fiber is in natural bijection with the set of enriched structures on $C/\Spec k$.
\end{corollary}

\section{A compactification}
\label{sec:compact}
For the remainder of this section, we fix $r \in \N_{\geq 1}$. In particular $r$ is not $\infty$. In this section, we will give a compactification $\ARLC_r$ of $\RLC_r$, in Definition~\ref{def:arlc}. This will be a stack, which will be called $\ARLC_r$ for ``weakly $r$-rich log curves'', such that the corresponding map between the algebraic stacks with log structures $\Min(\RLC_r) \to \Min(\ARLC_r)$ is an open dense immersion with $\Min(\ARLC_r)/\Min(\LC)$ proper. We will prove the following part of the second main theorem.

\begin{theorem}
\label{thm:comp}
For $r \in \N_{\geq 1}$. The stack $\RLC_r$ has a canonical modular compactification $\ARLC_r$, a log blowup and a sub-CFG of $\LC$, which is itself represented by an algebraic stack $\Min(\ARLC_r)$. The map $\Min(\RLC_r) \to \Min(\ARLC_r)$ is an open, dense inclusion.
\end{theorem}

For this, we will give a definition of weakly $r$-richness. First we need an adequate generalisation of $r$-closeness, which we give in the following lemma.

\begin{definition}
Let $M$ be a monoid and $A = \{a_1,\dots,a_n\}$ a subset of $M$. Then for $r \in \N$, the set $A$ is called \emph{weakly $r$-close} if for any sequence $(\lambda_i)_{i=1}^n$ of divisors of $r$ the ideal generated by all $\lambda_i a_i$ is principal, or equivalently there is a smallest element amongst all $\lambda_i a_i$.
\end{definition}

\begin{remark}
\label{rem:contravariancer}
Note that a set that is weakly $r'$-close is also weakly $r$-close for $r \mid r'$. A warning: this is similar to the case for $r$-close sets, except there the implication goes the other way. A set that is $r$-close is $r'$-close when $r \mid r'$.
\end{remark}

\begin{definition}
Let $C/\Spec k$ be a log curve over an algebraically closed field and let $\Gamma$ be the corresponding tropical curve. Then the edges of $\Gamma$ have a length in $\Mbar_k$, given by a function $\ell: E \to \Mbar_k$. Now let $c$ be a cut of $\Gamma$. Then $C/\Spec k$ is called \emph{weakly $r$-rich} if for any cut in $\Gamma$, the set of edge lengths of the edges in $c$ is weakly $r$-close. A log curve $C/S$ is called \emph{weakly $r$-rich} if it is weakly $r$-rich over all strict geometric points.
\end{definition}

\begin{example}
Any $r$-rich log curve is weakly $r$-rich. In fact, any $\infty$-rich log curve is weakly $r$-rich.
\end{example}
\begin{example}
For $r = 1$, this means that for every cut there is an edge with the smallest length among all edges in that cut. In other words, the ideal generated by the edge lengths of the edges in a cut is principal.
\end{example}
\begin{example}
We consider an weakly $1$-rich log structure on the $2$-gon over $\Spec k$. Note that there is only one cut of the $2$-gon, cutting both edges at once. The weakly richness then means that the labels of the two edges should be comparable. That means the map $f: \N^2 \to \Mbar_k$ should factor through the map $i_1: \N^2 \to \N^2, (a,b) \mapsto (a+b,b)$ or $i_2: \N^2 \to \N^2, (a,b) \mapsto (a,a+b)$. It can also factor through both, if and only if $f(1,0) = f(0,1)$. Note that even if it factors through $i_1$, the resulting map $\N^2 \oplus k^* \to k$ is not necessarily a log structure: if the labels of the two edges are the same in $\Mbar_k$, then it is only a pre-log structure. We see that every log structure factors uniquely through either one of the two log structures $\N^2 \oplus k^*$ where one of the edges has as label a unit vector and the other has $(1,1)$, or through the log structure $\N \oplus k^*$ where they both have label $1$. In the next example we will see that we get a bijection between minimal weakly $1$-rich log curves and $\P^1$, with the bijection sending a log structure to the ``ratio'' of their labels. Note that $\P^1$ is also the log blowup of the log point with log structure $\N^2$ in the ideal generated by the two unit vectors.
\end{example}
\begin{example}
We consider the universal deformation $C_U/U$ of the $2$-gon over $k[[u,v]]$ with its minimal log structure. This is not weakly rich, and in fact there are no weakly rich log structures on $C_U/U$, as the labels of the edges over the closed point are already incomparable in the structure sheaf, let alone in the log structure sheaf. But we can characterise when $C/S$, the pullback along a morphism $S \to U$, is weakly rich. Let $I$ be the monoid ideal generated by the labels of the two edges. By definition, $C/S$ is weakly rich if and only if the pullback of $I$ has a smallest element (under the $\leq$ relation). This is exactly the universal property of the log blowup of $U$ in $I$, so weakly rich curves coming from a pullback to $U$ are represented by a log blowup of $U$. The underlying scheme of this log blowup is $k[[u,v]]$ blown up in the origin.
\end{example}

\begin{remark}
For weakly rich curves, there is no equivalent reformulation of weakly richness using the terminology of circuit-connected components. The problem is that circuit-connected components do not behave well under contraction maps, and one circuit-connected component may split into several. Indeed, using that insight one can prove immediately that the CFG of log curves where on strict geometric points every circuit-connected component has a smallest edge is not compact. Cuts behave better under contraction maps, as we will see later in Lemma~\ref{lem:graphcuts}.
\end{remark}

\begin{definition}
\label{def:arlc}
Define $\ARLC_r/\LogSch$ to be the sub-CFG of $\LC/\LogSch$ consisting of \textbf{W}eakly $r$-\textbf{R}ich \textbf{L}og \textbf{C}urves.
\end{definition}

We have maps between these spaces, as described in the following two lemmas.
\begin{lemma}
\label{lem:rlctoarlc}
Let $r \in \N_{\geq 1}$. Any $\infty$-rich log curve is weakly $r$-rich. There are natural monomorphisms
\begin{equation}
\label{eq:rlcrarlcr}
\RLC_r \subset \ARLC_r \subset \LC
\end{equation}
and more generally
\begin{equation}
\label{eq:rlcr'arlcr}
\RLC_{r'} \subset \ARLC_r \subset \LC
\end{equation}
for any $r' \in \N_{\geq 1} \cup \{\infty\}$.
\end{lemma}

\begin{remark}
By Proposition~\ref{prop:richopen}, the first monomorphism in \eqref{eq:rlcr'arlcr} is always log open. We will later see that the corresponding map on algebraic stacks with a log structure is not an open immersion. The problem here is that minimal $r'$-rich log curves will generally not be \emph{minimal} weakly $r$-rich log curves unless $r = r'$. The fact that minimal $r$-rich log curves are minimal weakly $r$-rich log curves is the main content of Theorem~6.1 (together with the fact that $\ARLC_r$ is a log blowup).
\end{remark}

\begin{lemma}
\label{lem:contravariancer}
Let $r,r' \in \N_{\geq 1}$ with $r \mid r'$. Any weakly $r'$-rich log curve $C/S$ is weakly $r$-rich. There is a monomorphism
\begin{equation}
\ARLC_{r'} \subset \ARLC_{r} \subset \LC.
\end{equation}
\end{lemma}
\begin{proof}
This follows immediately from Remark~\ref{rem:contravariancer}.
\end{proof}

\begin{remark}
By Lemma~\ref{lem:covariancer}, Lemma~\ref{lem:rlctoarlc} and Lemma~\ref{rem:contravariancer} we get the following diagram of maps for a pair of positive integers $r \mid r'$.
\[
\begin{tikzcd}
\RLC_r \arrow[r]\arrow[d] & \RLC_{r'}\arrow[d] \\
\ARLC_r & \ARLC_{r'}\arrow[l]
\end{tikzcd}
\]
where all the maps are inclusions of a full subcategory. Again, we point out that the two horizontal maps go in opposite directions.
\end{remark}

We will prove that $\ARLC_r$ is a log blowup of $\LC$ using nuclear charts, as defined in \cite{poiret2020}. We start with a few general lemmas on blowups, and then, given a nuclear chart $\Uf$ of $\LC$, write down an explicit ideal of the characteristic monoid such that $\Uf \times_{\LC} \ARLC_r$ is exactly this blowup of $\Uf$ in this ideal.

\begin{definition}
A \textit{nuclear chart} is a nuclear curve of $\MLC$ (see Definition 3.39 of \cite{poiret2020}). This automatically comes with a strict \'etale map to $\LC$.
\end{definition}
In fact, all nuclear charts of $\LC$ together form a strict \'etale cover. So to show that the map $\ARLC_r \to \LC$ is a log blowup, we need to show that this is the case for nuclear charts, in a way compatible with refinements of nuclear charts. The following propositions show that exactly this is the case. We start with two general lemmas on respectively cuts and blowups.
\begin{lemma}
\label{lem:graphcuts}
Let $\Gamma = V \sqcup E$ and let $\Gamma \to \Gamma'$ be a contraction in a set of edges $S \subset E$. Then we have an inclusion $\Cuts(\Gamma') \subset \Cuts(\Gamma)$, with
\[
\Cuts(\Gamma') = \{c \in \Cuts(\Gamma) \mid c \cap S = \varnothing\}.
\]
\end{lemma}

\begin{proposition}
\label{prop:blowup-product}
Let $S$ be a log scheme and $I_1,I_2$ two monoid ideal sheaves. Then the blowup $B$ of $\Bl_{I_2} S$ in the pullback of $I_1$ is equal to the blowup $B' = \Bl_{I_1I_2}$.
\end{proposition}
\begin{proof}
We check that they satisfy the same universal property. The log scheme $B$ satisfies the universal property that for any $T \to B$, the pullback of $I_1$ and $I_2$ to $T$ locally have a smallest element; the log scheme $B'$ satisfies the universal property that for any $T \to B'$, the pullback of $I_1I_2$ locally has a smallest element. Clearly, if $I_1$ and $I_2$ locally have a smallest element, then $I_1 I_2$ locally has a smallest element. Conversely, if $I_1 I_2$ locally has a smallest element $ab$ with $a \in I_1, b \in I_2$, then by $ab \leq ac$ for $c \in I_2$ we see $b$ is a smallest element of $I_2$ and similarly $a$ is a smallest element of $I_1$. 
\end{proof}

Now we can explicitly write down the ideal we want to blow up in.

\begin{definition}
\label{def:igamma}
Let $\Uf \to \LC$ be a nuclear chart. Let $u$ be a point of the closed stratum and $\Gamma$ the corresponding tropical curve. Then define the log ideal $I_\Gamma$ of $\Uf$ to be
\[
\prod_{c = \{e_1,\dots,e_n\} \in \Cuts(\Gamma)} \prod_{\lambda_1,\dots,\lambda_n \mid r} (\ell(e_1)^{\lambda_1}, \dots, \ell(e_n)^{\lambda_n})
\]
where $\ell(e_i)$ denotes the length of edge $e_i$ (which is defined up to multiplication with units).
\end{definition}

\begin{proposition}
\label{def:igammaworks}
Let $\Uf \to \LC$ be a nuclear chart with corresponding graph $\Gamma$.
Then $\Uf \times_{\LC} \ARLC$ is the blowup $\Uf$ in the ideal $I_\Gamma$.
\end{proposition}
\begin{proof}
Denote the blowup of $\Uf$ in $I_\Gamma$ by $\pi: \Vf \to \Uf$.
Let $f: S \to \Uf$ be a morphism. Then the claim is that the natural curve $C_S/S$ we get is weakly $r$-rich if and only if $S$ factors through the blowup $\pi$. It is enough to do this for the case that $S$ is a strict geometric point of $\Uf$, as both weakly $r$-richness and factoring through the blowup can be decided on strict geometric points. So from now on, we assume that $S$ is $\Spec k$ with $k = \overline{k}$, and we let $\Gamma_S$ denote the dual graph of $C_S$. Note that $\Gamma_S$ is a contraction of $\Gamma$, and in particular $E(\Gamma_S)$ is a subset of $E(\Gamma)$. We use the notation $\ell_S,\ell_\Uf,\ell_\Vf$ for the length of an edge in respectively $S,\Uf$ and $\Vf$.

Assume $S \to \Uf$ factors through $\pi$. Let $c' = \{e_1,\dots,e_n\}$ be a cut of $\Gamma_S$ and $\lambda_i \mid r$ positive integers. Then by expanding the contracted edges, $c'$ induces a cut $c$ of $\Gamma$, so we see $(\ell_S(e_1)^{\lambda_1}, \dots, \ell_S(e_n)^{\lambda_n})$ is the pullback $f^* (\ell_\Uf(e_1)^{\lambda_1}, \dots, \ell_\Uf(e_n)^{\lambda_n})$. Since $f$ factors through $\pi$, and the pullback of $\pi^* (\ell_\Uf(e_1)^{\lambda_1}, \dots, \ell_\Uf(e_n)^{\lambda_n})$ is principal by the universal property of the log blowup, we have that $(\ell_S(e_1)^{\lambda_1}, \dots, \ell_S(e_n)^{\lambda_n})$ is principal. As this is true for any cut, we have that $C_S$ is weakly $r$-rich.

Now assume $C_S$ is weakly $r$-rich. We will prove that $S \to \Uf$ factors through the blowup. By the universal property it is enough to have that $f^* (\ell(e_1)^{\lambda_1}, \dots, \ell(e_n)^{\lambda_n})$ is principal for any cut $c = \{e_1,\dots,e_n\}$ of $\Gamma$ and $\lambda_i \mid r$. We consider two possibilities: either an edge in $c$ is contracted in $\Gamma_S$, or none are. In the first case, we see that $f^* (\ell(e_1)^{\lambda_1}, \dots, \ell(e_n)^{\lambda_n})$ is trivial and hence principal. In the second case, $c$ induces a cut $c'$ of $\Gamma_S$, and $f^* (\ell_\Uf(e_1)^{\lambda_1}, \dots, \ell_\Uf(e_n)^{\lambda_n}) = (\ell_S(e_1)^{\lambda_1}, \dots, \ell_S(e_n)^{\lambda_n})$, which is principal by our weakly $r$-richness assumption.
\end{proof}

To show that these ideals are compatible with refinement, we mainly need to show that cuts behave well with respect to contractions.

\begin{proposition}
\label{prop:logidealscomp}
Given a refinement $f: \Uf' \to \Uf$ of nuclear charts, with graphs $\Gamma',\Gamma$, we have $f^* I_\Gamma = I_{\Gamma'}$.
\end{proposition}
\begin{proof}
Note that we have a specialisation morphism $\Gamma \to \Gamma'$, which contracts those edges whose length becomes $0$. We then can use Lemma~\ref{lem:graphcuts} to classify the cuts of $\Gamma$ into cuts meeting a contracted edge, and cuts of $\Gamma'$.
Then for any cut $c$ of $\Gamma$ and any choice of $\lambda_i | r$, if the specialisation morphism $\Gamma \to \Gamma'$ contracts an edge in the cut, then $f^* (\ell_\Uf(e_1)^{\lambda_1}, \dots, \ell_\Uf(e_n)^{\lambda_n})$ is the unit ideal. And if $c$ maps to a cut $c'$ under the contraction morphism, then $f^* (\ell_\Uf(e_1)^{\lambda_1}, \dots, \ell_\Uf(e_n)^{\lambda_n}) = (\ell_{\Uf'}(e_1)^{\lambda_1}, \dots, \ell_{\Uf'}(e_n)^{\lambda_n})$. This completes the proof.
\end{proof}

\begin{example}
\label{ex:subdiv}
Consider the curve of arithmetic genus $1$ with dual graph $\Gamma$ a triangle, with its universal deformation $C$ over $S = \Spec k[[x,y,z]]$ with labels of edges $x,y,z$. We give $S$ the log structure induced by the divisor $xyz =0$ and $C$ the log structure induced by the pullback of this divisor. Then this is in fact a nuclear chart. We fix $r = 1$. In this case the ideal $I_\Gamma$ we blow up in is the ideal $(x^2y,x^2z,y^2x,y^2z,z^2x,z^2y,xyz)$. As $xyz$ is contained in the convex hull of the other generators, this is the same as the blowup in the ideal $(x^2y,x^2z,y^2x,y^2z,z^2x,z^2y)$. This blowup satisfies the universal property that the pullback of this ideal is principal.

There is a corresponding story to be told on the level of rational polyhedral cone complexes. Letting $M$ denote the monoid $\N^3$, the scheme $S$ corresponds to the rational cone $\sigma = \Hom(M,\R) = \R_{\geq 0}^3$. Then subcones of subdivisions of $\sigma$ correspond to sub CFGs of the CFG $S$ satisfying certain (in)equalities. For example, the ray $\R_{\geq 0} (1,1,1)$ corresponds to the sub CFG where all three edges have the same length. Now because all subsets of size $2$ are in a cut, the sub CFG $\ARLC_1 \times_\LC S$ of $S$ is defined by the fact that the edges are pairwise comparable. That defines a subdivision of the cone, as shown in the Figure~\ref{fig:subdiv}, where the subdivision has been intersected with the hyperplane $\{(e_1,e_2,e_3) \mid e_1 + e_2 + e_3 = 1\}$.

Here the subdivision is given by the planes $a=b,a=c,b=c$. That means that on one side $a \geq b$, and on the other side $a \leq b$, so $a$ and $b$ are comparable. Note that the union of the vertices in this picture correspond exactly to the sub CFG $\RLC_1 \times_\LC S$.

One can also see the compatibility with generalisation in this picture. Contracting a certain edge in a tropical curve corresponds to setting a certain edge length to $0$. On the cone side, this corresponds to intersecting with $\R_{\geq 0}^3$ with $\R_{\geq 0}^2$, or the subdivision of the triangle in Figure~\ref{fig:subdiv} with one of the edges of the triangle. With this perspective one gets exactly the subdivision of $\R^2$ corresponding to the $2$-gon, also shown in Figure~\ref{fig:trop}.
\end{example}
    
    
    



\begin{example}
We continue Example~\ref{ex:subdiv} for general $r$. We describe the corresponding subdivision of $\R_{\geq 0}^3$. Since every size two subset of the three edges is a cut, the subdivision is induced by the two-dimensional planes $\lambda e_i = \lambda' e_j$ for $i,j \in \{1,2,3\}$ distinct and for $\lambda,\lambda'$ divisors of $r$. See Figure~\ref{fig:sub-rtwo} for a picture in the case $r = 2$.
\end{example}
\begin{figure}[ht]
\begin{subfigure}{.49\textwidth}
\centering
  \begin{tikzpicture}[thick,scale=.56]
    \begin{scope}[every node/.style={draw=none, inner sep =-0pt}]
    \node (B) at (90:6) {};
    \node (C) at (210:6) {};
    \node (D) at (330:6) {};
    
    \draw (B) -- (C) node (BC) [midway] {};
    \draw (B) -- (D) node (BD) [midway] {};
    \draw (C) -- (D) node (CD) [midway] {};
    
    \draw (B) -- (CD);
    \draw (C) -- (BD);
    \draw (D) -- (BC);
    
    \node (A) at (0:0) {};
    \end{scope}

    \node at (56:2.1) {};
    \node at (124:2.1) {};
    \node at (180:2.1) {};
    \node at (236:2.2) {};
    \node at (304:2.2) {};
    \node at (0:2.1) {};

  \end{tikzpicture}
\caption{A subdivision of $\R_{\geq 0}^3$ intersected with the plane $e_1 + e_2 + e_3 = 1$.}
\label{fig:subdiv}
\end{subfigure}
\begin{subfigure}{.49\textwidth}
  \centering
  \begin{tikzpicture}[thick,scale=0.56]
    
    \begin{scope}[every node/.style={draw=none, inner sep =-0pt}]
    \node (B) at (90:6) {};
    \node (C) at (210:6) {};
    \node (D) at (330:6) {};

    \def\x{1};
    
    \draw [shorten <= -\x pt, shorten >= -\x pt] (B) -- (C) node (BC) [midway] {} node (BBC) [pos = .333] {} node (BCC) [pos = .666] {};
    \draw [shorten <= -\x pt, shorten >= -\x pt] (B) -- (D) node (BD) [midway] {} node (BBD) [pos = .333] {} node (BDD) [pos = .666] {};
    \draw [shorten <= -\x pt, shorten >= -\x pt] (C) -- (D) node (CD) [midway] {} node (CCD) [pos = .333] {} node (CDD) [pos = .666] {};

    \node (1) at (0:0) {};
    \node (2) at (intersection of  B--CD and C--BBD){};
    \node (3) at (intersection of  B--CD and C--BDD){};
    \node (4) at (intersection of  C--BD and B--CCD){};
    \node (5) at (intersection of  C--BD and B--CDD){};
    \node (6) at (intersection of  D--BC and B--CDD){};
    \node (7) at (intersection of  D--BC and B--CCD){};

    \draw [shorten <= -\x pt, shorten >= -\x pt] (B) -- (CD);
    \draw [shorten <= -\x pt, shorten >= -\x pt] (B) -- (CCD);
    \draw [shorten <= -\x pt, shorten >= -\x pt] (B) -- (CDD);
    
    \draw [shorten <= -\x pt, shorten >= -\x pt] (C) -- (BD);
    \draw [shorten <= -\x pt, shorten >= -\x pt] (C) -- (BBD);
    \draw [shorten <= -\x pt, shorten >= -\x pt] (C) -- (BDD);
    
    \draw [shorten <= -\x pt, shorten >= -\x pt] (D) -- (BC);
    \draw [shorten <= -\x pt, shorten >= -\x pt] (D) -- (BBC);
    \draw [shorten <= -\x pt, shorten >= -\x pt] (D) -- (BCC);

    \end{scope}

  \end{tikzpicture}\caption{A subdivision of $\R_{\geq 0}^3$, intersected with the plane $e_1 + e_2 + e_3 = 1$}
  \label{fig:sub-rtwo}
\end{subfigure}
\caption{}
\label{fig:trop5}
\end{figure}

Proposition~\ref{def:igammaworks} and Proposition~\ref{prop:logidealscomp} together form the following proposition.

\begin{corollary}
\label{cor:blowup}
The stack $\ARLC_r$ is the log blowup of $\LC$ in the logarithmic ideal descended from the $I_\Gamma$.
\end{corollary}

\begin{proof}[Proof of Theorem~\ref{thm:comp}]
We have already shown that $\ARLC_r$ has a monomorphism from $\RLC_r$, and that is represented by a log blowup $\Min(\ARLC_r)$ of $\Min(\LC)$. Log blowups are always proper (\cite[Theorem~2.6.3.3]{ogus}), so the stack $\Min(\ARLC_r)$ is proper over $\Min(\LC)$. 

Next, we will prove the map $\Min(\RLC_r) \to \Min(\ARLC_r)$ is an open embedding. By Proposition~\ref{prop:richopen} and \cite[Theorem~B.2]{wise2016}, we have the following open embeddings of algebraic stacks
\[
\begin{tikzcd}
\Min(\RLC_r)  \arrow[d,hook] & \Min(\ARLC_r) \arrow[d,hook]\\
\Log(\RLC_r) \arrow[r,hook] & \Log(\ARLC_r) \\
\end{tikzcd}
\]
The map $\Min(\RLC_r)  \to \Min(\ARLC_r) $ is defined by mapping a minimal $r$-rich log curve to the unique minimal weakly $r$-rich log curve lying over it, and a priori, filling in this map does not make the diagram commute. However, we claim that a minimal $r$-rich log curve already is a minimal weakly $r$-rich log curve. To show this, it is enough to consider the case of a curve $C$ over a geometric point $S = \Spec k$ with the dual graph having only one circuit-connected component. Let $\Gamma$ be the corresponding tropical curve with vertices $V$ and edges $E$, and let $\Mbar = \N^E$ denote the characteristic monoid of the minimal log structure on $S$. Then we consider a minimal $r$-rich log structure $M_r$, together with the unique minimal weakly $r$-rich log structure $f: M_a \to M_r$ lying over it. Then the claim is that $f$ is an isomorphism. It is enough to check this on characteristic monoids. By the assumption that there is one circuit-connected component, we have that $\Mbar_r = \N$. We also have a map $g: \Mbar \to \Mbar_r$, denoting the edge lengths. For every edge $e$ we have that $g(e) \in \N$ is a divisor of $r$. Also, $g$ factors through a map $h: \Mbar \to \Mbar_a$.

Note that for any cut $c \subset E$, by weakly $r$-richness of $\Mbar_a$, there is a minimal element $x$ among $A = \{\frac{r}{g(e)}h(e) \mid e \in c\} \subset \Mbar_a$. Then for all $x' \in A$, we have $x' \geq x$ or equivalently $x'-x \in \Mbar_a$. However, by construction $\bar{f}$ is constant with value $r$ on the set $A$, so $\bar{f}(x'-x) = 0$. But by construction $\bar{f}: \Mbar_a \to \Mbar_r$ is a sharp morphism, so we have $x' = x$ and $A$ consists of a single element. Equivalently, $r/f(e)g(e)$ is constant. Then as $\Mbar_a$ and $\Mbar_r$ are both fs and minimal we see that $\Mbar_c = \Mbar_r$ proving the claim.

That means that the above diagram with $\Min(\RLC_r) \to \Min(\ARLC_r)$ filled in is commutative, and hence $\Min(\RLC_r) \to \Min(\ARLC_r)$ is an open embedding.

Lastly, we need that the image of $\Min(\RLC_r)$ in $\Min(\ARLC_r)$ is dense. This follows from all our stacks having the stack of smooth curves as a dense open.
\end{proof}

Furthermore, we will prove the stack $\Min(\ARLC_1)$ is in fact smooth. For this we first need the following lemma on graphs.

\begin{lemma}
\label{lem:trees}
Let $G = (V,E)$ be a graph. Let $P$ be the set of partial orders on $E$ such that every cut of $G$ has a smallest element, and let $\leq$ be a minimal element of $P$ (i.e., such that there is no strictly weaker inequality also in $P$). Then every circuit-connected component has a unique minimal edge, and for every other edge there is a unique largest edge smaller than it.
\end{lemma}
\begin{proof}
We will prove this with induction on $|E|$, with the base case being trivial.

For the general case, we use that cuts always lie in a single circuit-connected component to reduce to the case where $G$ is circuit-connected. Then, as any two edges always lie in a cut together, for every two edges there is an edge that is smaller or equal to both. In particular, there is a smallest edge $m$. Clearly, every cut that contains $m$ already has a smallest element. The cuts that do not contain $m$ correspond to the cuts of the graph $G'$ where $m$ is contracted. This graph $G'$ is not necessarily circuit-connected, but by the induction hypothesis we know that this lemma holds for $G'$. Then adding the relation $m \leq e$ for the minimal $e \in E(G')$ proves our claim also holds for $G$.
\end{proof}
\begin{example}
Figure~\ref{fig:tree} contains an example of Lemma~\ref{lem:trees}. The graph $G$ with edge set $E = \{a,b,c,d,e,f\}$ is described in Figure~\ref{fig:tree-one}. If we want a partial order on $E$ such that every cut has a smallest element, it is both necessary and sufficient that $\{a,b,c\}$ contains a smallest element and all of $\{d,e,f\}$ are pairwise comparable. One such choice might be $a \leq b, c$ and $d \leq e \leq f$. The corresponding description of $\leq$ is shown in Figure~\ref{fig:tree-two}.
\begin{figure}[ht]
\begin{subfigure}{.49\textwidth}
  \centering
  \begin{tikzpicture}[thick,scale=.8]
    \begin{scope}[every node/.style={circle, draw,fill=black!50,inner sep=0pt, minimum width=4pt}]
    \node (A) at (0,0) {};
    \node (B) at (3,0) {};
    \node (C) at (6,0) {};
    \node (D) at (4.5,2.59807621135) {};
    \end{scope}

    \draw (A) -- (B) -- (C) -- (D) -- (B);
    \def\x{60};
    \draw (A) to[out=-\x,in=180+\x] (B);
    \draw (A) to[out=\x,in=180-\x] (B);
    
    \node at (1.5,1.1) {$a$};
    \node at (1.5,-.3) {$b$};
    \node at (1.5,-1.1) {$c$};

    \node at (4.5, -.3) {$d$};
    \node at (3.5, 1.5) {$e$};
    \node at (5.5, 1.5) {$f$};

  \end{tikzpicture}
  \caption{The graph $G$}
  \label{fig:tree-one}
\end{subfigure}
\begin{subfigure}{.49\textwidth}
  \centering
  \begin{tikzpicture}[thick,scale=1.4]

    \node (a) at (2,0) {$a$};
    \node (b) at (1.2,1) {$b$};
    \node (c) at (2.8,1) {$c$};

    \node (d) at (4,0) {$d$};
    \node (e) at (4,1) {$e$};
    \node (f) at (4,2) {$f$};

    \draw [-to](a) -> (b);
    \draw [-to](a) -> (c);

    \draw [-to](d) -> (e);
    \draw [-to](e) -> (f);

  \end{tikzpicture}
  \caption{The corresponding forest $A$}
  \label{fig:tree-two}
\end{subfigure}
\caption{}
\label{fig:tree}
\end{figure}
\end{example}

\begin{corollary}
\label{cor:locfree}
Let $C/\Spec k$ be a minimal weakly $1$-rich log curve over an algebarically closed field. Then the characteristic monoid $\Mbar_k$ is free.
\end{corollary}
\begin{proof}
Let $\Gamma$ denote the tropical curve corresponding to $C$, and let $E = \{e_1,\dots, e_m\}$ denote its edge set. We will explicitly describe the characteristic monoid $\Mbar_k$. We pick a choice function $f : \Cuts(\Gamma) \to E$ sending a cut to the edge with the smallest length (if there are multiple edges with the smallest length, then the one with the smallest index), and then we define $M_f \subset \Z^E$ to be the monoid generated by $\N^E$ and $e - f(c)$ for every cut $c$ and every edge $e\in c$. This is the characteristic monoid of a pre-log structure corresponding to $\Mbar_k$, so one has to divide out by the kernel of the structure morphism $M_f \to k/k^*$ to obtain $\Mbar_k$.

We first give an explicit formula for $M_f$. Consider the pre-order $\leq$ on $E$ generated by $f(c) \leq e$ for any cut $c$ and any edge $e \in c$. By Lemma~\ref{lem:trees} this pre-order can be described as follows. Every circuit-connected component has a unique smallest edge for $\leq$. We denote the set of these edges by $E_0$. Let $E_1 = E \setminus E_0$ be the set of edges that are not the smallest in their component. For every edge $e \in E_1$, there is a unique largest edge smaller than it. Denote this edge by $g(e) = \max\{e' : e' < e\}$. By Lemma~\ref{lem:trees} the preorder $\leq$ is generated by $g(e) \leq e$ for $e \in E_1$.

This in particular implies that for any cut $c$ and any edge $e \in c$ we have that the inequality $f(c) \leq e \in \Z^E$ can be written as a composition of inequalities \[f(c) = g^k(e) \leq \dots \leq g^2(e) \leq g(e) \leq e.\]
In particular, $e - f(c)$ is a non-negative linear combination of $e' - g(e')$ for $e' \in E_1$. Then we see that $M_f \sub \Z^E$ is generated by the elements $e$ for $e\in E_0$, and by the elements $e - g(e)$ for $e \in E_1$. In total $M_f$ has $|E|$ generators and rank $|E|$, hence it is free.

The characteristic monoid $\Mbar$ of the logification of this pre-log structure is the quotient of $M_f$ by the kernel of the morphism $M_f \to k/k^*$. The target is a sharp monoid, so the kernel is a face of $M_f$, hence the inclusion of the kernel in $M_f$ is isomorphic to the natural inclusion $\N^{n} \to \N^{|E|}$ for some $n \leq |E|$, so this quotient $\Mbar = \N^{|E|-n}$ is also free.
\end{proof}

We already knew the stack $\Min(\ARLC_1)$ is log \'etale over $\Min(\LC)$, as it is a blowup. We then get the following theorem.
\begin{theorem}
\label{thm:smooth}
The stack $\Min(\ARLC_1)$ is smooth.
\end{theorem}
\begin{proof}
We will use Theorem~III.3.3.1 of \cite{ogus}, which \'etale locally around a geometric point $\overline{x}$ writes the map $\Min(\ARLC_1) \to \Spec \Z$ as a composite of a strict log \'etale (and hence \'etale) map, and the map $f_\theta$ which is a base change of $\Z[P] \to \Z$, where $P$ is the characteristic monoid at $\overline{x}$. By Corollary~\ref{cor:locfree}, the monoid $P$ is free, hence $\Z[P] = \A^n_\Z$ for some $n$, which is smooth over $\Z$. So in total, we get a composite of smooth maps, hence $\Min(\ARLC_1)$ is smooth over $\Spec \Z$.
\end{proof}

\begin{remark}
For $r > 1$, the equivalent of Corollary~\ref{cor:locfree} does not hold. We give a counterexample. Let $s$ be the largest proper divisor of $r$. Consider the curve $C/k$ with dual graph a triangle with edge lengths $a,b,c \in \Mbar_k$. We can choose $\Mbar_k$ to be the universal monoid satisfying $b \geq rc, rb \geq a \geq sb$. This automatically satisfies $a \geq rc$. By construction the curve is then minimal weakly $r$-rich. This monoid has extremal rays $(r,1,0),(r^2,r,1),(rs,1,r)$ and $(s,0,1)$. As $\Mbar^\gp = \Z^3$, we see that $\Mbar$ is not free (in fact, its dual $\Hom(\Mbar, \R_{\geq 0})$ is not even simplicial).

For $r = 2$ we have $s = 1$, and $\Mbar_k$ has as dual $\Hom(\Mbar, \R_{\geq 0})$ the unique (up to permutation of $a,b,c$) non-simplicial cone in Figure~\ref{fig:sub-rtwo}.

As a consequence, $\Min(\ARLC_r)$ is not smooth for $r > 1$.
\end{remark}

\subsection{\texorpdfstring{The stack of $r$-rich log curves}{The stack of r-rich log curves}}
\label{subsec:rrichagain}
Now we have shown that $\Min(\RLC_r)$ embeds into $\Min(\ARLC_r)$, we can easily show the final part of Theorem~\ref{thm:mainrlc}
\begin{theorem}
\label{thm:rlcsmoothalgebraic}
The stack $\Min(\RLC_r)$ is algebraic and smooth.
\end{theorem}
\begin{proof}[Proof of Theorem~\ref{thm:rich}]
By Theorem~\ref{thm:comp}, the stack $\Min(\RLC_r)$ is an open inside an algebraic stack, hence algebraic.

For smoothness, note that the characteristic monoid at a geometric point of $\Min(\RLC_r)$ is free by Proposition~\ref{def:richminimal}. Then Theorem~III.3.3.1 of \cite{ogus} states that \'etale locally around a geometric point $\overline{x}$ we can write the map from $\Min(\RLC_r)$ to $\Spec \Z$ as a composite of a strict log \'etale (and hence \'etale) map, and the map $f_\theta$ which is a base change of $\Spec \Z[P] \to \Spec \Z$, where $P$ is the characteristic monoid at $\overline{x}$. As $P$ is free, $f_\theta$ is the map $\A^n_\Z \to \Spec \Z$ for some $n$, and hence is smooth, so the map $\Min(\RLC_r) \to \Spec \Z$ is smooth. 
\end{proof}

\section{Comparison to work by Holmes-Biesel}
\label{sec:compbiesel}
There has already been a more involved definition of a moduli space of enriched structures on curves over general schemes in \cite{biesel2019}, and in this section we prove that this is the same as $\Min(\RLC_1)$. In Theorem~5.9 of \cite{biesel2019}, it is proven that their stack of enriched structures on curves is isomorphic to a certain substack of a universal N\'eron-model-admitting stack. We will give an isomorphism of $\Min(\RLC_1)$ to the same stack.

We recall a definition from \cite{holmes2015}.
\begin{definition}[Definition~10.1, \cite{holmes2015}]
\label{def:neronmodel}
Let $C/S$ be a nodal curve over an algebraic stack, smooth over a dense open substack $U \subset S$. Write $J$ for the Jacobian of $C_U/U$. This is an abelian scheme over $U$. A N\'eron-model-admitting-morphism for a curve $C/S$ is a morphism $f: T \to S$ of algebraic stacks such that
\begin{enumerate}
  \item $T$ is regular;
  \item $U \times_S T$ is dense in $T$;
  \item $f^* J$ admits a N\'eron model over $T$.
\end{enumerate}
\end{definition}
In Corollary~10.3 of \cite{holmes2015}, they define the stack $\widetilde{\Mcal}$ over $\Min(\LC)$, the universal N\'eron-model admitting morphism. This stack is not quasi-compact. They define when a curve is $r$-strongly aligned (not invariant under pullback), define $\widetilde{\Mcal}^{\leq r}$ as the terminal $r$-strongly aligned curve, and show that these are N\'eron-model-admitting over $\Min(\LC)$, and in fact $\widetilde{\Mcal} = \bigcup_r \widetilde{\Mcal}^{\leq r}$. This definition of being $r$-strongly aligned in the case $r = 1$ closely resembles our notion of $1$-richness, and the notion of being $r$-strongly aligned for some $r$ resembles our notation of being $\infty$-rich, except for the fact that their definitions do not work with log structures, and cannot be checked just on geometric points. Nevertheless, the resulting universal stacks $\widetilde{\Mcal}^{\leq 1}$ and $\widetilde{\Mcal}$ from their paper are in fact equal as we will prove to $\Min(\RLC_1)$ and $\Min(\RLC_\infty)$ respectively.

\begin{lemma}
\label{lem:nmam}
For $r \in \N_{\geq 1} \cup \{\infty\}$ the natural map $\Min(\RLC_r) \to \Min(\LC)$ is a N\'eron-model-admitting-morphism and hence factors uniquely through $\widetilde{\Mcal}$. For $r = 1$ this map factors through $\widetilde{\Mcal}^{\leq 1}$.
\end{lemma}
\begin{proof}
We need to prove exactly the three statements in \ref{def:neronmodel} for the universal curve over $\Min(\RLC_r)$. First, note that $\Min(\RLC_r)$ is log smooth over $\Spec \Z$ because for $r < \infty$ it is an open subset of a log blowup and for $r = \infty$ it has as an open cover by such sets.. In particular, it is log regular. Furthermore, note that the log structure on the base of a minimal $r$-rich log curve is by definition locally free. Then by Theorem~III.1.11.6 of \cite{ogus} the stack $\Min(\RLC_r)$ is also regular.

Then as $\Min(\RLC_r)$ is log regular, the locus where the log structure is trivial is dense. (This locus is exactly the locus where the curve is smooth.)

The third part is a bit more complicated, as a curve $C/S$ being aligned in the notion of \cite{holmes2015} (which is immediately the case if it is the underlying algebraic curve of a $\infty$-rich log curve) does not imply that a N\'eron model exists; there are some complications having to do with regularity. However, for this one can use \cite{poiret2020}, where they do take a logarithmic approach. There is a small change of notation; in \cite{holmes2015} N\'eron models are required to be separated, and in \cite{poiret2020} they do not. In \cite{poiret2020} they prove that (not necessarily separated) N\'eron models always exists for a curve over a log regular base. Furthermore by \cite[Theorem~8.3]{poiret2020}, as the universal curve over $\Min(\RLC_r)$ is log aligned, the N\'eron model is separated. Hence the third condition of a N\'eron-model-admitting-morphism is satisfied as well.

For the second part, by \cite[Definition~12.1]{holmes2015} the stack $\widetilde{\Mcal}^{\leq 1}$ is the largest substack of $\widetilde{\Mcal}$ where the pullback of the curve is regular. Again by Theorem~III.1.11.6 of \cite{ogus}, to prove that the universal curve over $\Min(\RLC_1)$ is regular, we just need to show that the stalks of the characteristic monoid are locally free, as log regularity follows immediately from log smoothness. We know by the local descriptions of log structures of log curves from \cite{kato2000} that, as the base of the curve is locally free, we only need to check locally freeness at the nodes. For a geometric point $\overline{s}$ where the tropical curve has edge set $E$ and circuit-connected components $T$, and for $e \in E$ contained in $t \in T$, we need to check whether the monoid $\N^2 \oplus_\N \N^T$ is free, where the map $\N \to \N^2$ is the diagonal and the map $\N \to \N^T$ sends $1$ to $1 \cdot t$ as the curve is $1$-rich. We see this monoid is isomorphic to $\N^{|T| + 1}$, so indeed the universal curve over $\Min(\RLC_1)$ is regular, and hence $\Min(\RLC_1) \to \widetilde{\Mcal}$ factors through $\widetilde{\Mcal}^{\leq 1}$.
\end{proof}

\begin{theorem}
\label{theorem:compbiesel}
The natural maps $\Min(\RLC_1) \to \widetilde{\Mcal}^{\leq 1}$ and $\Min(\RLC_\infty) \to \widetilde{\Mcal}$ are isomorphisms.
\end{theorem}
\begin{proof}
First, we will construct maps the other way. Note that by Theorem~1.2 of \cite{holmescrelle} N\'eron-model admitting morphisms are aligned (as in \cite{holmes2015}). With $j: U \to \Mf$ the embedding of the locus of smooth curves, that means the log structure $j_* \Ocal_U$ is log aligned (as in \cite{poiret2020}), hence this gives a map $\widetilde{\Mcal} \to \Min(\RLC_\infty)$. Also, for $1$-strongly aligned curves this log structure is $1$-rich, hence this map restricts to a map $\widetilde{\Mcal}^{\leq 1}\to \Min(\RLC_1)$.

We claim that $\widetilde{\Mcal}^{\leq 1}$ and $\widetilde{\Mcal}$ are regular and separated. To see this, we note that $\Min(\RLC_1)$ and $\Min(\RLC_\infty)$ are regular as by \ref{lem:nmam}, and $\Min(\RLC_r) \to \Min(\LC)$ for $r < \infty$ is an open embedding followed by a log blowup, and $\Min(\RLC_\infty)$ is the limit of separated stacks. Also, all mentioned morphisms are isomorphisms on the dense locus of smooth curves. Any morphism on a regular separated stack that is the identity on an open is the identity, which means that $\Min(\RLC_1) \to \widetilde{\Mcal}^{\leq 1}$ and $\Min(\RLC_\infty) \to \widetilde{\Mcal}$ have respectively $\widetilde{\Mcal}^{\leq 1}\to \Min(\RLC_1)$ and $\widetilde{\Mcal} \to \Min(\RLC_\infty)$ as inverses, and hence are isomorphisms. 
\end{proof}

\appendix

\section{\texorpdfstring{Minimal objects of $\RLC_r$}{Minimal objects of RLCr}}
\label{sec:appendix}
In this appendix we aim to characterise the minimal objects of $\RLC_r$, and prove that the conditions of the descent lemma, Theorem~\ref{thm:gillam}, hold. We start by defining a notion of a basic $r$-rich log curve, and will then show that an $r$-rich log curve is minimal if and only if it is basic.
\begin{definition}
\label{app:def:richbasic}
Let $C/\Spec k$ be a $r$-rich log curve over an algebraically closed field and let $\Gamma = (V, E)$ be the corresponding tropical curve. Let $T$ denote the set of circuit-connected components, for $t \in T$ let $a_t \in \Mbar_k$ be the root of $T$, and for $e \in t$ let $\lambda_e \mid r$ be such that the length of the $e$th edge is $\lambda_e a_t$. Then the natural map $\N^E \to \Mbar_k$ factors through the \emph{root map} $\N^E \to \N^T$ where $T$ is the set of circuit-connected components and $\N^E \to \N^T$ sends $e$ to $\lambda_e$, and $\N^T \to \Mbar_k$ sends $t$ to $a_t$. The log curve $C/\Spec k$ is called \emph{basic $r$-rich} if the resulting map $\N^T \to \Mbar_k$ is an isomorphism. An $r$-rich log curve $C/S$ is called \emph{basic $r$-rich} if it is basic $r$-rich over all strict geometric points.
\end{definition}

\begin{proposition}
\label{app:prop:richbasicinitial}
Let $\ul{C}/\ul{S}$ be a prestable curve, and consider am $r$-rich log curve structure $C/S$ with sheaves $(M_C,M_S)$ on it. Then there is a basic $r$-rich log curve structure $C'/S'$ on $\ul{C}/\ul{S}$ such that there is a map between the log curves $(C,S) \to (C',S')$ that is the identity on schemes. Both the basic log structure and the map are unique up to unique isomorphism. 
\end{proposition}
\begin{proof}
First, we recall some fact about maps between log structures. Let $M$ be a log structures on a scheme $X$. Then for $M' \to M$ a map of log structures, we have $M' = M \times_{\Mbar} \Mbar'$. This means $M' \to M$ is uniquely determined by its map on the characteristic monoids. Vice versa, if we have a map of sharp sheaves of monoids $\Mbar' \to \Mbar$, there is a unique lift $M \times_{\Mbar} \Mbar' \to M$. In other words, the category of maps of log structures to $M$ is equivalent to the category of sharp monoid maps to $\Mbar$.

Now as both being $r$-rich and being basic $r$-rich is defined on the level of characteristic monoids, we can reduce the proposition to the level of characteristic monoids. Define the category $GS$ having objects $(\Mbar_S'/\N^E, \beta)$ with $\beta: \Mbar_S' \to \Mbar_S$ a map of sharp monoids over $\N^E$. It remains to show this category has an initial object. Locally, this is immediate, as we can take $\Mbar_S' = \N^T$ with $\N^E \to \Mbar_S'$ the root map. Then by uniqueness and uniqueness of isomorphisms, this construction glues to a global construction.

\end{proof}

This proposition on the prevalence of basic curves, also sometimes phrased as the set of basic $r$-rich log curves being \emph{weakly terminal} in the category of $r$-rich log curves, turns out to be enough to both characterise the minimal objects and to fulfil the first condition of the descent lemma Theorem~\ref{thm:gillam}.

\begin{lemma}
\label{app:lem:weaklyterminalminimal}
Let $X$ be a category, and $W$ a subcategory such that for each object in $X$ there is a map to an element of $W$, unique up to unique isomorphism. Then an object in $X$ is minimal if and only if this map is an isomorphism.
\end{lemma} 
\begin{proof}
We start by proving that an element $w \in W$ is minimal. We do this by directly verifying the definition. Let's say we have three objects $x_1,x_2,x$ where $x$ is in $W$, and we have maps $x_2 \to x_1, x_2 \to x$. Now all these three objects have a map to an element of $W$, but by the property of $W$ used on $x_2$, all these objects must be $x$, and we get a commutative diagram
\[\begin{tikzcd}
  & x & \\
x_1 \arrow[ru] &        & \arrow[lu] x \\
  & x_2 \arrow[ru] \arrow[lu] &  \\
\end{tikzcd}\]
The arrow $x \to x$ has to be an isomorphism, which provides us with the necessary map $x_1 \to x$. Again as $W$ is weakly terminal, such a map must be unique, hence $x$ is minimal.

Now that we have proven that an object of $W$ is minimal, it remains to prove that any map between two minimal objects is an isomorphism, as we already know that any object admits a map to an object in $W$ and is minimal if this map is an isomorphism. So if we have such a map $f: w \to z$, then by minimality applied to the to maps $f: w \to z, \id: w\to w$ the map $f$ has a unique left-inverse $g: z \to w$. Then again, $g$ has a unique left-inverse $h$. But now $f = (h \circ g) \circ f = h \circ (g \circ f) = h$, so $g$ is a two-sided inverse of $f$ and $f$ is an isomorphism.
\end{proof}

From this proposition, we immediately get the connection between basic $r$-rich log curves and minimal objects.
\begin{corollary}
\label{app:prop:richmineqbasic}
An $r$-rich log curve is minimal if and only if it is basic.
\end{corollary}

\begin{proposition}
\label{app:prop:mrlc}
The category $\RLC_r$ satisfies the conditions of the descent lemma (Theorem~\ref{thm:gillam}).
\end{proposition}
\begin{proof}
The first condition, of every $r$-rich log curve admitting a map to a basic $r$-rich log curve, is directly given by Proposition~\ref{app:prop:richbasicinitial}.

It remains to show that the pullback of a basic $r$-rich log curve along a map of log schemes $f$ is basic if and only if $f$ is strict. Note that every morphism is by definition the composition of a strict morphism and a morphism lying over the identity of the algebraic scheme of the source. Then it suffices to prove that the pullback of a basic $r$-rich log curve along a strict morphism is basic, and the pullback along a non-strict morphism lying over the identity of the algebraic scheme is not basic. For the first of these statements, note that given a strict morphism $f: U \to V$, any strict geometric point of $U$ is a strict geometric point of $V$, and hence any condition on strict geometric points of $V$ will also hold for strict geometric points of $U$. In particular includes basic $r$-richness. For the second of these statements, we have by definition that over a geometric point, a map from a basic $r$-rich log structure to a $r$-rich log structure is an isomorphism if and only if both log structures are basic.
\end{proof}

\bibliographystyle{hep}
\addcontentsline{toc}{section}{References}
\bibliography{references.bib}


\end{document}